\documentclass[12pt,reqno]{amsart}

\usepackage{amsfonts}
\usepackage{amscd}
\usepackage{amssymb}
\usepackage{amsfonts}
\usepackage{amscd}
\usepackage{amsmath} 
\usepackage{mathrsfs}
\usepackage{dsfont}
\usepackage{enumerate}
\usepackage{xcolor}
\usepackage{float} 
\usepackage[pdftex]{graphicx}
\allowdisplaybreaks
\usepackage{url}
\usepackage{soul}
\usepackage[hidelinks]{hyperref} 
 \hypersetup{
     colorlinks,
     linkcolor={red!70!black},
     citecolor={blue!80!black},
     urlcolor={blue!60!black}
 }

\date{\today}

\newtheorem*{theorem*}{Theorem}
\newtheorem{theorem}{Theorem}[section]
\newtheorem{corollary}[theorem]{\bf{Corollary}}
\newtheorem{lemma}[theorem]{Lemma}
\newtheorem{proposition}[theorem]{Proposition}
\theoremstyle{definition}

\theoremstyle{remark}
\newtheorem{remark}[theorem]{\bf{Remark}}

\numberwithin{equation}{section}

\newcommand{\beas}{\begin{eqnarray*}}
\newcommand{\eeas}{\end{eqnarray*}}
\newcommand{\bes} {\begin{equation*}}
\newcommand{\ees} {\end{equation*}}
\newcommand{\be} {\begin{equation}}
\newcommand{\ee} {\end{equation}}
\newcommand{\bea} {\begin{eqnarray}}
\newcommand{\eea} {\end{eqnarray}}

\newcommand{\R}{\mathbb R}
\newcommand{\C}{\mathbb C}
\newcommand{\Z}{\mathbb Z}%
\newcommand{\N}{\mathbb N}

\newcommand{\X}{\mathbb{X}}
\newcommand{\M}{\mathcal{M}}
\newcommand{\what}{\widehat}

\newcommand{\ol}{\overline}
\newcommand{\mf}{\mathfrak}

\renewcommand{\Im}{\text{Im}}
\renewcommand{\Re}{\text{Re}}

\renewcommand{\Re}{\operatorname{Re}}
\renewcommand{\Im}{\operatorname{Im}}
 
\usepackage{color}

\newcommand{\abs}[1]{\lvert#1\rvert}

\def \ee {\end{equation}}

\numberwithin{equation}{section}

\title[Weighted Fourier inequalities on symmetric spaces]{Weighted Fourier inequalities and application of restriction theorems on rank one Riemannian symmetric spaces of noncompact type}

\author[Kumar, Pusti, Rana, Singh]{Pratyoosh Kumar, Sanjoy Pusti, Tapendu Rana, Mandeep Singh}

\address{Pratyoosh Kumar,
\endgraf Department of Mathematics,
IIT Guwahati,
\endgraf Guwahati - 781039, Assam, India.}
\email{kumar.pratyoosh@gmail.com, pratyoosh@iitg.in}
\address{Sanjoy Pusti \endgraf Department of Mathematics, IIT Bombay, \endgraf Powai, Mumbai-400076, India.}
\email{spusti@gmail.com, sanjoy@math.iitb.ac.in}

 \address{Tapendu Rana  \endgraf Department of Mathematics: Analysis, Logic and Discrete Mathematics,	\endgraf Ghent University,  Krijgslaan 281, Building S8, B 9000 Ghent, Belgium.} \email{tapendurana@gmail.com, tapendu.rana@ugent.be}

\address{Mandeep Singh,
\endgraf Department of Mathematics,
IIT Guwahati,
\endgraf Guwahati - 781039, Assam, India.}
\email{mandeepgodara27@gmail.com, mandeep.singh@iitg.ac.in}

\date{}

\begin{document}
\begin{abstract}
This article explores weighted $(L^p, L^q)$ inequalities for the Fourier transform in rank one Riemannian symmetric spaces of noncompact type. We establish both necessary and sufficient conditions for these inequalities to hold. To prove the weighted Fourier inequalities, we apply restriction theorems on symmetric spaces and utilize Calder{\'o}n's estimate for sublinear operators. While establishing the necessary conditions, we demonstrate that Harish-Chandra's elementary spherical functions play a crucial role in this setting. Furthermore, we apply our findings to derive Fourier inequalities with polynomial and exponential weights.

\end{abstract}
\maketitle

\section{Introduction}
Weighted inequalities for the Fourier transform provide a natural measure to characterize both uncertainty and balance between function growth and smoothness.  These weighted and unweighted inequalities are not only essential in harmonic analysis but also find applications in partial differential equations, probability theory, and other areas of mathematics.
Classical results like the Hausdorff-Young inequality have long illustrated the relationship between the norms of a function and its Fourier transform, establishing a foundational framework for more advanced inequalities. Pitt's inequality \cite{Pitt_37} builds on this framework by incorporating weights that precisely account for the distribution of a function and its Fourier transform, providing a more nuanced understanding of their relationship. More precisely, the problem of characterizing pairs of weights $(u,v)$ governing strong-type norm inequalities of the following type
\begin{equation}\label{weight_Rn_int}
 \left( \int_{\R^N} |\mathcal{F}(f)(\xi)|^q u(\xi)^{q} \, d\xi \right)^{\frac{1}{q}}\leq C \left( \int_{\R^N} |f(x)|^{p} v(x)^p  \,dx\right)^{\frac{1}{p}}
\end{equation}
is of considerable importance in analysis, where, as weight, we refer to non-negative measurable functions and $\mathcal{F}f $ denotes the classical Fourier transform of $f $, defined by
\begin{align*}
    \mathcal{F}f(\xi) =\int_{\R^{N}}f(x) e^{-i \langle \xi, x\rangle} \,dx.
\end{align*}
 The main goal of this article is to obtain necessary and sufficient conditions on weights for which analogs of \eqref{weight_Rn_int} hold in the context of rank one Riemannian symmetric spaces of noncompact type, which includes all hyperbolic spaces.

In the Euclidean context, when we take power weights $u_{\sigma}(\xi):= |\xi|^{-\sigma}$ and $v_{\kappa}(x):= |x|^{\kappa}$, inequality \eqref{weight_Rn_int} transforms into the classical Pitt's inequality introduced by H.R. Pitt in 1937 \cite{Pitt_37}. This inequality for power weights was subsequently generalized by Stein \cite{St_56} and  Benedetto and Heinig \cite{BH_03}. The problem of extending it to more general weights has been extensively studied since the mid-1970s, beginning with Muckenhoupt's paper \cite{Mu_78}, where he formulated this problem for the Fourier transform and derived some sufficient conditions. 
Later, in 1983, Muckenhoupt \cite{Mu_83_T}, as well as Jurkat and Sampson \cite{JS_84}, gave different but equivalent conditions (see \cite{Mu_83_P}) for the pair of weights \(u\) and \(v\) for which \eqref{weight_Rn_int} holds, when \(p\) and \(q\)   not both equal to 2. Heinig \cite{He_84} independently obtained the same result by proving the following theorem. 

Let $u^*$ denote the non-increasing rearrangement of a function $u$ in $\mathbb{R}^N$, and $p' = p/(p-1)$ be the conjugate exponent of $1 \leq p \leq \infty$. The result of Heinig \cite[Theorem 3.1]{He_84} can be stated as follows:
\begin{theorem}\label{thm_heinig}
    Let $1\leq p\leq q\leq \infty$. If $p<\infty$ and $p,q$ not both equal to $2$, then 
    \begin{align}\label{suf_p,q}
        \sup\limits_{s>0} \left( \int_0^{1/s} u^*(t)^q\, dt \right)^{\frac{1}{q}} \left( \int_0^{s} \left(1/v \right)^*(t)^{-p'} \, dt \right)^{\frac{1}{p'}} <\infty
    \end{align}
    implies \eqref{weight_Rn_int} holds. If $p=q=2$, then \eqref{suf_p,q} and
    \begin{align}\label{suf_p=q=2}
        \sup\limits_{s>0} \left( \int_0^{1/s}u^*(t)^2 \,t^{-1}\,  dt \right)^{\frac{1}{2}} \left( \int_0^{s} \left(1/v \right)^*(t)^{-2} \, t^{-1} \, dt \right)^{\frac{1}{2}} <\infty
    \end{align}
    imply \eqref{weight_Rn_int}.
\end{theorem}
Conversely, under additional monotonicity assumptions on the weights, it can be shown that \eqref{suf_p,q} is also necessary \cite{Mu_83_T, He_84}. More generally, the following result on necessary conditions for radial weights in the Euclidean setting is known; see \cite{He_84, DLDS_17}.
\begin{theorem}\label{thm_nec_int_RN}
  Suppose that the inequality \eqref{weight_Rn_int} holds for $1<p,q<\infty$, and  $u$ and $v$ are two radial weights in $\R^N$. Then we have the following
\begin{equation}\label{nec_org_RN}
    \sup_{s>0} \left(\int_{|\xi|<{\theta_0/s}} u(\xi)^q\, d\xi \right)^{\frac{1}{q}} \left(\int_{|x|<s} v(x)^{-p'}\,   dx\right)^{\frac{1}{p'}}< \infty,
\end{equation}
where $\theta_0$ is any positive number less than $q_{N/2-1}$, the first zero of Bessel function $J_{N/2-1}$. In particular, $q_{N/2-1} \geq \pi/2$.
\end{theorem}
 We would like to mention that one can generalize this result further by obtaining necessary conditions even for non-radial weights as well, for which we refer to \cite{La_94, Be_11, DLDS_17}. However, for Fourier inequalities involving radial weights, if $u$ and $v$ are radial with $u(|\cdot|)$ being non-increasing and $v(|\cdot|)$ being non-decreasing, then \eqref{nec_org_RN} is both necessary and sufficient for the validity of \eqref{weight_Rn_int}. Specifically, when $u(\xi)$ and $v(x)$ are  power weights of the form $u_{\sigma}(\xi) = |\xi|^{-\sigma}$ and $v_{\kappa}(x) = |x|^{\kappa}$, we arrive at the classical Pitt's inequality
\begin{equation}\label{Clss_Pitt_Rn_pw}
 \left( \int_{\R^N} |\mathcal{F}(f)(\xi)|^q |\xi|^{-q\sigma}  \, d\xi \right)^{\frac{1}{q}}\leq C \left( \int_{\R^N} |f(x)|^{p}  |x|^{p \kappa} dx \right)^{\frac{1}{p}},
\end{equation}
with $1<p \leq q<\infty$, which is valid if and only if
\begin{equation}
    \begin{aligned}\label{k-s=N}
        \kappa-\sigma & = N\left(1-\frac{1}{p}-\frac{1}{q} \right),\\
        0\leq \sigma <\frac{N}{q}, \qquad &\text{ and } \qquad 0\leq \kappa< \frac{N}{p'}; 
    \end{aligned}
\end{equation}
see \cite{Pitt_37, St_56}, and \cite{BH_03}. Since the inequality \eqref{Clss_Pitt_Rn_pw} is one of the classical results in Fourier analysis, specifically encompassing the Hausdorff-Young inequality ($q=p'\geq 2$, $\sigma=\kappa=0$), Paley's inequality ($1<p=q \leq 2$, $\kappa=0$) and Hardy-Littlewood inequality ($2\leq p=q<\infty$, $\sigma=0$), among others,  it has continued to be a subject of extensive research by numerous authors, resulting in a substantial body of literature: see, for instance, \cite{JT_70, BH_92, La_94, BH_03, Be_08} and the reference within. Notably, in the last decade, Tikhonov and his collaborators have significantly advanced the study of weighted Fourier inequalities; see, for example, \cite{LT_12, DGT_13, GIT_16, DLDS_17}. 

Inequalities of the type \eqref{weight_Rn_int} are not only of interest by themselves but also have significant applications; see \cite{BL_94}. They describe the balance between the relative sizes of a function and its transform at infinity and can be regarded as a quantitative expression of the uncertainty principle \cite{Be_95, Jo_16, GIT_16}. Additionally, these inequalities play an important role in restriction theory  \cite{BH_92, BS_92}.

This article contributes to the literature by establishing weighted $(L^p, L^q)$ Fourier inequalities or Pitt's inequalities in the context of Riemannian symmetric spaces $\X = G/K$, where $G$ is a real rank one noncompact connected semisimple Lie group with finite center, and $K$ is a maximal compact subgroup of $G$. Additionally, our results and techniques are applicable to harmonic $NA$ groups, also known as Damek-Ricci spaces. We establish both sufficient and necessary conditions for the validity of these inequalities. Furthermore, we apply our results to obtain Fourier inequalities with polynomial and exponential weights, investigating how the exponential volume growth of the Riemannian symmetric spaces influences the behavior of these inequalities.  For any other unexplained notation, we refer the reader to Section \ref{Preliminaries}.

From this point onward,  $u$ and $v$ will represent two non-negative measurable functions defined over $\mathbb{R}$ and $\mathbb{X}$, respectively. The non-increasing rearrangement of $u$ (on $\R$) with respect to the Plancherel measure $|c(\lambda)|^{-2} d\lambda$ will be denoted by $U = u^{\star}$, and the non-increasing rearrangement of $v$ (on $\X$) with respect to the Haar measure will be denoted by $v^*$. We let $1/V = (1/v)^*$. We now present our first main results of this article, which can be viewed as an analog of Theorem \ref{thm_heinig}.
\begin{theorem}\label{thm_Pitt_r_int}
Let $1 \leq p \leq q \leq \infty$. If $p < \infty$ and $p, q$ are not both equal to 2, and
\be \label{uv_loc_int}
\sup_{s>0} \left(\int_{0}^{s^{-1}} U(t)^q \, dt\right)^{\frac{1}{q}}\left(\int_{0}^s V(t)^{-p'} \, dt\right)^{\frac{1}{p'}} < \infty,
\ee
then  there exists a constant $C>0$ such that for all $f \in C_c^{\infty}(\X)$, the following holds
\be \label{pitt_rad_int}
\left(\int_{\R} \left( \int_K\left|  \widetilde{f}(\lambda,k) \right|^{2} dk\right)^{\frac{q}{2}}    u(\lambda)^q |{c(\lambda)}|^{-2} \,d\lambda\right)^{\frac{1}{q}} \leq C\left(\int_{{\X}}{\left|  f(x)\right| }^p  v(x)^p \,dx\right)^{\frac{1}{p}}. 
\ee
If p = q = $2$ then \eqref{uv_loc_int}  and
\be \label{uv_glo_pitt_int}
\sup_{s>0}\left(\int_{s^{-1}}^{\infty} U(t)^2 t^{-1}dt\right)^{\frac{1}{2}}\left(\int_s^{\infty} V(t)^{-2}t^{-1} dt\right)^{\frac{1}{2}} < \infty,
\ee
imply $(\ref{pitt_rad_int})$.
\end{theorem}
In the theorem above, we consider the weighted Fourier inequality by treating the Fourier transform $\widetilde{f}$ as a function on $\mathbb{R} \times K$. However, it is important to note that the holomorphic extension of the Fourier transform of an $L^p(\X)$ function $(1 \leq p < 2)$ is perhaps the most distinctive feature of Riemannian symmetric spaces (which we will refer to as symmetric spaces) of noncompact type. Recognizing this characteristic, we investigate weighted Fourier inequalities where the Fourier transform depends not only on $\mathbb{R}$ but also on non-real $\lambda$. Consequently, the following result does not have any counterpart in $\mathbb{R}^N$.
\begin{theorem}\label{thm_Pitt_nr_intn}
   Let $1<q_0<\infty $ be fixed and $\rho_{q_0} = (2/{q_0}-1)\rho$. Suppose that for a given $1 \leq p \leq  q \leq \infty$, the functions $U$ and $V$ satisfy the following inequalities:
    \begin{equation} \label{uv_loc_nr_int1}
\sup_{s>0}\left(\int_{0}^{s^{-1}} U(t)^q dt\right)^{\frac{1}{q}}\left(\int_{0}^s V(t)^{-p'} dt\right)^{\frac{1}{p'}} < \infty,
\end{equation}
and 
\begin{equation}\label{uv_glo_nr_int1}
     \sup_{s>0}\left(\int_{s^{-1}}^{\infty} U(t)^q t^{-\frac{q}{\max\{ q_0,q_0'\}}}dt\right)^{\frac{1}{q}}\left(\int_s^{\infty} V(t)^{-p'}t^{-\frac{p'}{\max\{q_0,q_0'\}}} dt\right)^{\frac{1}{p'}} < \infty.
\end{equation}
Then there exists a constant $C>0$ such that for all $f \in C_c^{\infty}(\X)$, the following holds
\be\label{pitt_nrad_int}
\left(\int_{\R}\left( \int_K\left|  \widetilde{f}(\lambda +i\rho_{q_0},k) \right|^{q_0} dk\right)^{\frac{q}{q_0}} u(\lambda)^q |{c(\lambda)}|^{-2}\, d\lambda\right)^{\frac{1}{q}} \leq C\left(\int_{\X} |f(x)|^{p} v(x)^p \,  dx\right)^{\frac{1}{p}}.
\ee
\end{theorem}
A critical aspect of our study involves deriving the necessary conditions for weighted Fourier inequalities in these non-Euclidean settings. These conditions are essential for understanding the limitations and scope of such inequalities. The unique exponential volume growth properties of symmetric spaces necessitate modifications to the classical conditions known in Euclidean settings. One key feature in this context is the role of Harish-Chandra's elementary spherical function $\phi_{\lambda}$, which plays a central role in the analysis of symmetric spaces.
\begin{theorem}\label{thm_nec_int}
Let $q_0 \in [1,\infty]$ be fixed. Assume that  $u$ and $v$ are two radial weights such that the Pitt's inequality \eqref{pitt_nrad_int} holds for $1<p,q<\infty$. Then the following condition is necessary
\begin{equation}\label{nec_org}
    \sup_{s>0} \left(\int_{|\lambda|<\theta_0/s} u(\lambda)^q |c(\lambda)|^{-2} \, d\lambda\right)^{\frac{1}{q}} \left(\int_{|x|< s}v(x)^{-p'}  \phi_{i\rho_{q_0}}(x)^{p'}  \,dx\right)^{\frac{1}{p'}}< \infty,
\end{equation}
where $\theta_0$ is any positive number less than $\pi/2$.
\end{theorem}
Our next objective is to establish a particular type of weighted Fourier inequality. The Hausdorff-Young inequality states that the Fourier transform maps functions from $L^p$ spaces to $L^{p'}$ spaces with respect to the Plancherel measure for $1 \leq p \leq 2$. Paley's inequality addresses the following question: Given $p \in [1, 2)$, can we find a measure $\nu$ such that the Fourier transform operator is bounded from $L^p$ to $L^p(\nu)$? Specifically, H{\"o}rmander\cite[Theorem 1.10]{Hormander_Acta_60} proved the following in the context of Euclidean spaces, which also can also be seen as a weighted version of Plancherel’s formula for $L^p(\mathbb{R}^N)$ with $1 < p \leq 2$.
\begin{theorem}
Let $u$ be a positive function in $L^{1, \infty}(\R^N)$. Then, for any $ 1<p\leq 2$, we have
    \begin{align*}
        \left( \int_{\R^N} |\mathcal{F}f(\xi)|^{p} u(\xi)^{2-p} \,d\xi \right)^{\frac{1}{p}}\leq C_p \|f\|_{L^p(\R^N)},
    \end{align*}
where the constant $C_p$ depends only on $p$ and $u$.
\end{theorem}
We provide the following analog of the theorem above in the context of symmetric spaces of noncompact type with rank one.
\begin{theorem}\label{thm_paley}
    Let $\X$ be a rank one symmetric space of noncompact type with dimension $n\geq 3 $. Assume that $u$ be a positive function in $L^1\left(\R, |c(\lambda)|^{-2}d\lambda\right)$.
Then for $1\leq p\leq 2$, we have 
    \begin{align}\label{paley's}
         \left(  \int_{\R} \int_{K}  \left| \widetilde{f}(\lambda+i\rho_p , k)\right|^p    
  u(\lambda)^{2-p} |c(\lambda)|^{-2}\, dk  \,d\lambda \right)^{\frac{1}{p}} \leq C_p \|u\|_{L^1\left(|c(\lambda)|^{-2}\right)}^{\frac{2}{p}-1} \|f\|_{L^p(\X)}.
    \end{align}
\end{theorem}
Before we delve into the specifics of this article, let us compare our analysis of weighted Fourier inequalities in $\X$ rank one symmetric space of noncompact type with that of Euclidean settings. In Euclidean settings,  to prove Pitt's inequality, authors (\cite{He_84, GLT_18}) generally use the Plancherel theorem and the strong type $(1,\infty)$ boundedness of the classical Fourier transform operator $\mathcal{F}$, which follows from the boundedness of the kernel $e^{-i \langle \xi, x \rangle}$ used to define the classical Fourier transform. However, this analogy breaks down in symmetric spaces because the (Poisson) kernel used to define the Helgason Fourier transform $\widetilde{f}$ is not bounded anymore. This necessitates a different approach, focusing on restriction inequalities of the form
\begin{align}\label{eqn_res}
\| \widetilde{f}(\lambda, \cdot) \|_{L^q(K)} \leq C \| f \|_{L^p(\X)}.
\end{align}
As observed by Ray and Sarkar in \cite{RS_09} (see also \cite{KRS_10}), due to the analytic continuation of the Fourier transform in symmetric spaces, the exponent $q$ in \eqref{eqn_res} depends on the imaginary part of $\lambda$, which is rooted in the \textit{Kunze-Stein phenomenon}. Consequently, the imaginary part of $\lambda$ in \eqref{eqn_res} also dictates the exponents for which the Hausdorff-Young inequality holds in $\X$ (see Theorem \ref{thm_RS_HYinq}). To utilize these inequalities, for a given $q_0 \in (1,\infty)$, we take $\Im \lambda = \rho_{q_0}$ in \eqref{eqn_res}, enabling us to assert that the following operator
\begin{align*}
\mathcal{T}_{q_0} f(\lambda) := \| \widetilde{f}(\lambda + i \rho_{q_0}, \cdot) \|_{L^{q_0}(K)}
\end{align*}
is of strong type $(1, \infty)$. We then use the Hausdorff-Young inequality and the Calderón estimate \cite{Cal} for sublinear operators to prove Theorem \ref{thm_Pitt_nr_intn}. Next, we investigate the relationship between the two hypotheses \eqref{uv_loc_nr_int1} and \eqref{uv_glo_nr_int1} in Theorem \ref{thm_Pitt_nr_intn}. Results of this type are related to Hardy's weighted inequalities and were also explored by Heinig \cite{He_84}. We observe that, using the monotonicity properties of $U$ and $V$ in some cases (depending on $p$ and $q$), \eqref{uv_loc_nr_int1} implies \eqref{uv_glo_nr_int1} (see Lemma \ref{1em_imply_new}), thereby generalizing Heinig's result \cite[Proposition 2.6]{He_84}. Furthermore, using the same result, Theorem \ref{thm_Pitt_r_int} will follow as a corollary of Theorem \ref{thm_Pitt_nr_intn}.

Proving the necessary condition for the weighted Fourier inequality \eqref{pitt_nrad_int} constitutes a major challenge. In the Euclidean setting, the authors in \cite{He_84, DLDS_17} used the fact that the classical Fourier transform $\mathcal{F}(f)$ for a radial function $f$ reduces to the Hankel transform, which integrates $f$ against a Bessel function. Since the zeros of Bessel functions are well-known, one can use this fact and appropriately choose a compactly supported function to prove the necessary condition \eqref{nec_org_RN} in $\R^{N}$.  On the other hand, in symmetric spaces of noncompact type, for radial functions, the Helgason Fourier transform $\widetilde{f}$ reduces to spherical transform $\what{f}$ defined by
\begin{align*}
    \what{f}(\lambda) = \int_{\X} f(x) \phi_{\lambda} (x^{-1})\, dx.
\end{align*}
While $\phi_{\lambda}$ can be expressed in terms of Bessel functions when $|x|$ is small (see \cite[Theorem 2.1]{Stanton_Tomas}), unlike in the Euclidean setting, this formula expresses $\phi_{\lambda}$ as a sum of several Bessel functions. Thus, it is not useful for analyzing its zeros. Consequently, the analogous argument may fail in this setting. To overcome this obstacle, we employ a different strategy. Instead of finding the zeros of $\phi_{\lambda}$, we utilize an integral representation \eqref{int_rep_phi} of $\phi_{\lambda}$ by Koornwinder \cite{Ko_75}. This formula enables us to obtain a lower bound for $|\widehat{f}(\lambda + i \rho_{q_0})|$ for any compactly supported non-negative radial function $f$ on $\X$ (see Lemma \ref{lem_f^geq}). This bound, in turn, leads to the necessary condition \eqref{nec_org}. Notably, the necessary condition does not require the assumption $p \leq q$.  This result reveals both similarities and differences between the behavior of spherical analysis and its Euclidean counterpart. Additionally, it highlights the analogous behavior of $\phi_{\lambda}$ and $e^{i\langle \xi, \cdot \rangle}$ despite their fundamental differences.

Observing the necessary condition \eqref{nec_org} and recalling that $\phi_{i \rho_{q_0}}$ for $q_0 \in [1,\infty]$ exhibits exponential decay towards infinity (see \eqref{est_phi_{rho_p}}), we have considered both polynomial and exponential weights on the function side. These exponential weights naturally arise in spaces with exponential growth and play a role analogous to $A_p$ weights in these settings (see \cite{OR_22, GRS_23}). In both the necessary and sufficient cases, sharp estimates of Harish-Chandra's elementary spherical function $\phi_{\lambda}$ and $c$-function are crucial. Finally, to prove Paley's inequality, we utilize Stein's analytic interpolation theorem, in contrast to Marcinkiewicz's interpolation in the Euclidean setting.\\

We conclude this section by providing an outline of this article. In the next section, we present the necessary background on harmonic analysis on semisimple Lie groups and symmetric spaces of noncompact type and recall some previously known results relevant to our study. In Section \ref{sec_Pitt_on_X}, we prove our results on weighted Fourier inequalities, namely Theorem \ref{thm_Pitt_r_int} and Theorem \ref{thm_Pitt_nr_intn}. Section \ref{sec_nec} delves into the necessary condition for the weighted Fourier inequalities to hold and proves Theorem \ref{thm_nec_int}. In Section \ref{sec_example}, we identify sufficient and necessary conditions for polynomial and exponential weights to ensure that weighted Fourier inequalities hold. Finally, in Section \ref{sec_paley}, we prove Theorem \ref{thm_paley}, thereby concluding our investigation.

\section{{Preliminaries}} \label{Preliminaries}
\subsection{Generalities}
The letters $\N$, $\Z$, $\R$, and $\C$ will respectively denote the set of all natural numbers, the ring of integers, and the fields of real and complex numbers. For $ z \in \C $, we use the notations $\Re z$ and $\Im z$ for real and imaginary parts of $z$, respectively. We shall follow the standard practice of using the letters $C$, $C_1$, $C_2$, etc., for positive constants, whose value may change from one line to another. Occasionally, the constants will be suffixed to show their dependencies on important parameters. Throughout this article, we use $X\lesssim Y$ or $Y\gtrsim X$ to denote the estimate $X\leq CY$ or $X\geq CY$ for some absolute constant $C>0$. We shall also use the notation $X\asymp Y$ for $X\lesssim Y$ and $Y\lesssim X$. For any Lebesgue exponent $ p \in [1,\infty]$, let $p^{\prime}$ denote the conjugate exponent such that $\frac{1}{p} + \frac{1}{p'}=1$, with obvious assumption $p^{\prime}=\infty$ when $p=1$ and vice-versa. 
   
   \subsection{Riemannian symmetric spaces} 
   
Here, we review some general facts and necessary preliminaries regarding semisimple Lie groups and harmonic analysis on Riemannian symmetric spaces of noncompact type. Most of this information is already known and can be found in \cite{GV, He_00, Helgason_GA}. Our notation is also standard, as seen in \cite{Stanton_Tomas, Io_02}. To make this article self-contained, we will gather only the results used throughout this paper. 
   
Let $ G $ be a noncompact connected semisimple real rank one Lie group with finite centre, with its Lie algebra $\mathfrak g$. Let $ \theta $ be a Cartan involution of $ \mathfrak{g} $ and $ \mathfrak{g} =\mathfrak{k+p} $ be the associated Cartan decomposition. Let $ K=\exp \mathfrak{k} $ be a maximal compact subgroup of $ G $ and let $ \X =G/K $ be an associated symmetric space with origin $ \textbf{0} =\{K\} $. 
Let $B$ denote the Cartan Killing form of $\mathfrak{g}$. It is known that $B|_{\mathfrak{p} \times \mathfrak{p}}$ is positive definite and hence induces an inner product and a norm $|\cdot|$ on $\mathfrak{p}$. The homogeneous space $\X=G/K$ is a smooth manifold. The tangent space of $\X$ at the point $o =eK$ can be naturally identified to $\mathfrak{p}$ and the restriction of $B$ on $\mathfrak{p}$ then induces a $G$-invariant Riemannian metric $d$ on $\X$. For $x\in \X$, we denote $|x|$ by \begin{equation*}
  |x|=d(x, eK).  
\end{equation*}
Let $ \mathfrak{a} $ be a maximal abelian subspace of $ \mathfrak{p} $. Since the group $ G $ is of real rank one,  dim $\mathfrak{a} =1 $.  Let $ \Sigma $ be the set of nonzero roots of the pair $ (\mathfrak{g,a}) $, and let $ W $ be the associated Weyl group. For rank one case, it is well known that either $ \Sigma =\{-\alpha,\alpha\} $ or $ \{-2\alpha,-\alpha,\alpha,2\alpha \} $, where $ \alpha $ is a positive root and the Weyl group $ W $ associated to $ \Sigma  $ is  \{-Id, Id\}, where Id is the identity operator. Let $ \mathfrak{a}^+ =\{ H \in \mathfrak{a} : \alpha(H)>0 \} $ be a positive Weyl chamber, and let $ \Sigma^+ $ be the corresponding set of positive roots.  In our case, $ \Sigma ^+ = \{\alpha\}$ or $ \{ \alpha, 2\alpha\} $. For any root $ \beta \in \Sigma  $, let $ \mathfrak{g}_\beta $ be the root space associated with $ \beta $. We let 
			\begin{align*}
			\mathfrak{n} =\sum_{\beta \in \Sigma^+} \mathfrak{g}_\beta, \qquad \text{and} \qquad  N =\exp  \mathfrak{n}.
		 \end{align*}
The group $ G $ has an Iwasawa decomposition \bes G= K (\exp \mathfrak{a})N,
			\ees and a Cartan decomposition 
			\bes G=K(\exp \mathfrak{a}^+) K.\ees This decomposition is unique. For each $ g \in G $, we denote $ H(g) \in \mathfrak{a} $
			and $ g^+ \in \ol {\mf{a}^+} $ are the unique elements such that
			\begin{align*}
			g=k \exp H(g) v, \qquad k\in K, v\in N,
			\end{align*} and 
\begin{align*}
g=k_1 \exp (g^+)k_2 \qquad k_1, k_2\in K.
\end{align*}	

Let  $H_0$ be the unique element in $\mathfrak a$ such that $\alpha(H_0)=1$ and through this we identify $\mathfrak a$ with $\R$ as $t\leftrightarrow tH_0$ and $\mathfrak a_+= \{H\in \mathfrak a\mid \alpha(H)>0\}$ is identified with the set of positive real numbers.   We also identify $\mathfrak a^*$ and its complexification $\mathfrak a^*_\C$ with $\R$  and $\C$ respectively by $t\leftrightarrow t\alpha$ and  $z\leftrightarrow z\alpha$, $t\in \R$, $z\in \C$. Let $A=\exp\mathfrak{a}=\left\{a_t:=\exp(t H_0)\mid t\in\R\right\}$ and $A^+=\left\{a_t\mid t>0\right\}$.
 Let $m_1=\dim \mathfrak g_\alpha$ and $m_{2}=\dim \mathfrak g_{2\alpha}$ where $\mathfrak g_\alpha$ and
 $\mathfrak g_{2\alpha}$ are the root spaces corresponding to $\alpha$ and $2\alpha$. We will also denote $n= m_1+m_2+1$ as the dimension of the symmetric space $\X$. As usual, let $\rho=\frac 12(m_1+2m_2)\alpha$ denote the half sum of the positive roots. By abuse of notation we will denote $\rho(H_0)=\frac 12(m_1+2m_2)$ by $\rho$. For  $p\in[1,\infty]$, we define 
\begin{equation*}
 \rho_p=\left({2}/{p}-1\right)\rho \quad\text{ and } \quad S_{p}=\{z\in\mathbb{C}:|\Im z|\leq|\rho_{p}|\}.   
\end{equation*}

Let $ dg, dk,$ and $  dv $ be the Haar measures on the groups $ G, K,$ and $ N$, respectively. We  normalize $ dk $ such that $ \int_K dk =1 $.  We have the following integral formulae corresponding to the Cartan decomposition respectively, which holds for any integrable function $f$:
 \begin{equation}\label{car_dec}
\int_Gf(g)dg=\int_K\int_{\R^+}\int_K f(k_1a_tk_2) \Delta(t)\,dk_1\,dt\,dk_2. 
\end{equation} 
where $\Delta(t)=(2\sinh t)^{m_1 + m_2}(2\cosh t)^{m_2}$.

A function $f$ on $G$ is said to be right $K$-invariant if $f(gk)=f(g)$ for all $g\in G, k\in K$. We identify a function on $\X=G/K$ as a right $K$-invariant function on $G$. When we refer to integration of a function $f$ on $\X$, we mean the integration of the function $f$ on $G$ as a
right $K$-invariant function.

\subsection{Fourier transform}
For a sufficiently nice function $f$ on $\X$, its {\em Helgason Fourier transform} $\widetilde{f}$ is a function defined on $\C \times K$ given by 
\be \label{defn:hft}
\widetilde{f}(\lambda,k) = \int_{\X} f(x) e^{(i\lambda- \rho)H(x^{-1}k)} \,dx,\:\:\:\:\:\: \lambda \in \C,\:\: k \in K, 
\ee
whenever the integral exists (\cite[p. 199]{Helgason_GA}). 

It is known that if $f\in L^1(\X)$ then $\widetilde{f}(\lambda, k)$ is a continuous function of $\lambda \in \R$, for almost every $k\in K$. If in addition $\widetilde{f}\in L^1(\R\times K, |c(\lambda)|^{-2}~d\lambda~dk)$ then the following Fourier inversion holds,
\be\label{hft}
f(gK)= |W|^{-1}\int_{\R} \int_K \widetilde{f}(\lambda, k)~e^{-(i\lambda+\rho)H(g^{-1}k)} ~ |c(\lambda)|^{-2}~dk\,d\lambda,
\ee
for almost every $gK\in \X$ (\cite[Chapter III, Theorem 1.8, Theorem 1.9]{Helgason_GA}), where $c(\lambda)$ is the Harish-Chandra's $c$-function given by \bes c(\lambda)=\frac{2^{\rho-i\lambda} \Gamma(\frac{m_1 + m_2 +1}{2})\Gamma(i\lambda)}{\Gamma(\frac{\rho+ i\lambda}{2}) \Gamma(\frac{m_1+ 2}{4} + \frac{i\lambda}{2})}.\ees
It is normalized such that $c(-i\rho)=1$. Moreover, $f \mapsto \widetilde{f}$ extends to an isometry from  $L^2(\X)$ onto $L^2(\R\times K, |c(\lambda)|^{-2}~d\lambda~dk )$ (\cite[Chapter III, Theorem 1.5]{Helgason_GA}):
\begin{align*}
    \int_{\X} |f(x)|^2 dx = |W|^{-1} \int_{\R} \int_{K}  |\widetilde{f} (\lambda, k) |^2 |c(\lambda)|^{-2}\,  dk\,d\lambda.
\end{align*}
A function $f$ is called $K$-biinvariant or radial if  for all $g\in G$,   $k_1, k_2\in K$ \bes f(k_1gk_2)=f(g) .
\ees  We denote the set of all $K$-biinvariant functions by $\mathcal F (G//K)$. By employing the Cartan decomposition, we can extend a function $f$ initially defined on $\overline{A^+}$ to a $K$-biinvariant function on $G$ by
\begin{equation}\label{A+_to_Kbi}
  	 f( k_1 a_t k_2 ) = f(a_t),
    \end{equation}
     \text{ for all } $k_1 , k_2 \in K, t\geq 0$. 
     
Let $\mathbb D(G/K)$ be the algebra of $G$-invariant differential operators on $\X$. The elementary spherical functions $\phi$ are $C^\infty$ functions and are joint eigenfunctions of all $D\in\mathbb D(G/K)$ for some complex eigenvalue $\lambda(D)$. That is $$D\phi=\lambda(D)\phi, D\in\mathbb D(G/K).$$They are parametrized by $\lambda\in\C$. The algebra $\mathbb D(G/K)$ is generated by the Laplace-Beltrami operator $ L$. Then we have, for all $\lambda\in\C, \phi_\lambda$ is a $C^\infty$ solution of  
\begin{equation}\label{phi-lambda}
L\phi=-(\lambda^2 + \rho^2)\phi.
\end{equation}
			For any $\lambda\in \C$  the elementary spherical function $\phi_\lambda$ has the following integral representation
\begin{equation}\label{int_rep}
    \phi_\lambda(x)=\int_K e^{-(i\lambda+\rho)H(xk)}\,dk \text{ for all } x\in G.
\end{equation}
The spherical transform $\what{f}$ of a  suitable $K$-biinvariant function $f$ is defined by the formula:
$$\what{f}(\lambda)=\int_{\X} f(x)\phi_\lambda(x^{-1})\,dx.$$
It is easy to check that for suitable $K$-biinvariant function $f$ on $G$; its Helgason Fourier transform $\widetilde{f}$ reduces to the spherical transform $\what{f}$.

We now list down some well-known properties of the elementary spherical functions that are important for us (\cite[Prop 3.1.4 and Chapter 4, \S 4.6]{GV}, \cite[Lemma 1.18, p. 221]{Helgason_GA}).

\begin{enumerate}
\item $\phi_\lambda(g)$ is $K$-biinvariant in $g\in G$,  $\phi_\lambda=\phi_{-\lambda}$, $\phi_\lambda(g)=\phi_\lambda(g^{-1})$.
\item $\phi_\lambda(g)$ is $C^\infty$ in $g\in G$ and holomorphic in $\lambda\in\C$.
\item$|\phi_\lambda(x)|\leq 1$ for all $x\in G$ if and only if $\lambda\in S_1$.
\item For all $\lambda\in \C$, we have
\begin{align}\label{phi_il}
|\phi_\lambda(g)| \leq  \phi_{i \Im \lambda}(g).
\end{align}
\item For $ \lambda \in \C  $ with $ \Im \lambda <0 $,  from the asymptotic estimate of $\phi_\lambda$ \cite{HC_58}, it follows that (see \cite[(2.6)]{PRS_11})
\begin{equation}\label{est_phi_lam}
|\phi_{\lambda}(a_t)| \asymp e^{-(\Im \lambda +\rho)|t|} . 					
\end{equation}
Particularly, taking $ \lambda = i \rho_p$, we have 
\begin{align}\label{est_phi_{rho_p}}
    \phi_{i\rho_p} (a_t) = \begin{cases}
        e^{-\frac{2\rho}{p'}|t|} \qquad &\text{if } 1\leq p<2,\\
        e^{-\frac{2\rho}{p}|t|} \qquad &\text{if } p>2.
    \end{cases}
\end{align}
\item We also have the following inequality for $p=2$ case, see \cite[(3.5)]{ADY_96}
\begin{align}\label{est_phi_0}
\phi_0(a_t) \asymp \left(1+t\right)e^{-\rho t}, \:\: t\geq 0.
\end{align}
\end{enumerate}		
We also recall the following formula of spherical function from  \cite[Lemma 2.2]{Stanton_Tomas}, see also \cite[(2.16)]{Ko_75}.
 \begin{lemma}\label{int_exp_phi}
     We have  for all $\lambda \in \C $, 
     \begin{equation}\label{int_rep_phi}
       \begin{aligned}
      &{(c_0 \sinh 2t)}^{-1}{\Delta(t)} \phi_{\lambda}(a_t) \\
      & = \int_0^t (\cosh 2t- \cosh 2s)^{\frac{m_2}{2}-1} \left(\int_0^s (\cosh s -\cosh r)^{\frac{m_1}{2}-1} \cos(\lambda r) \,dr\right) \sinh s\, ds,
   \end{aligned}
   \end{equation}where $c_0$ is a constant given by \begin{equation*}
       c_0=\pi^{1/2} 2^{m_2/2-2} \frac{\Gamma(\frac{n-1}{2})}{\Gamma(\frac{n}{2})}.
   \end{equation*}
 \end{lemma}
The authors Ray and Sarkar \cite{RS_09} established the following Hausdorff-Young inequality. We note that they proved the result for $n\geq 3$. However, the result is also valid for the case when $n=2$.
\begin{theorem}[{{\cite[Theorem 4.6]{RS_09}}}]\label{thm_RS_HYinq}
    For $1\leq p\leq 2$ and $p\leq q\leq p'$,
    \begin{align}\label{HY_ineq}
        \left( \int_{\R} \left( \int_K \left| \widetilde {f}(\lambda +i\rho_q, k)  \right|^q \, dk \right)^{\frac{p'}{q}} |c(\lambda)|^{-2} d\lambda \right) ^{\frac{1}{p'}} \leq C_{p,q} \|f\|_{L^p(\X)}.
    \end{align}
    The case  $p=q=2$ is a weakening of the Plancherel theorem. We also note that
when $q=p' $, the result best resembles the classical Hausdorff-Young inequality at
the lower boundary of the strip $S_p$.
\end{theorem}

\begin{theorem}[{{\cite[Theorem 4.2]{RS_09}}}]\label{RS_4.2}
    Let $1\leq p<2$. Then for $p<q<p'$ and for $\lambda \in \C$ with $\Im  \lambda  = \rho_q$, we have 
\begin{align*}
    \left( \int_{K} |\widetilde{f}(\lambda, k)|^q dk\right)^{\frac{1}{q}} \leq C_{p,q} \|f\|_{L^p(\X)},
\end{align*}
for all $f \in L^p(\X)$. Moreover, when $p=1$, then $q \in [1,\infty] $ and $C_{p,q}=1$.
\end{theorem}
Let us recall the following estimate \cite[Lemma 4.8]{RS_09} of $|c(\lambda)|^{-2}$, which can be obtained from \cite[Lemma 4.2]{Stanton_Tomas} the explicit expression of $|c(\lambda)|^{-2}$.  
\begin{lemma}\label{est_c-2}
    We have 
    \begin{align}\label{sharp_c-2}
        |c(\lambda)|^{-2} \asymp \lambda^2 (1+|\lambda|)^{n-3}, \quad \text{for } \lambda \in \R.
    \end{align}
\end{lemma}

    \subsection{Lorentz spaces}Let $(X,\mu) $ be a \textit{sigma}-finite measure space. For $f: X \rightarrow \C$  a measurable function on $ X $, the distribution function $ d_f $ defined on $ [0,\infty) $  is given by $$ d_f(\alpha) = \mu (\{ x \in X : |f(x)| > \alpha \}) .$$
   For  $ p,q\in [1,\infty]$, the Lorentz spaces $ L^{p,q}(X) $ consist of all measurable functions $f  $ on $X$ for which  $\| f \|_{p,q} $ is finite, where  $\| f \|_{p,q} $ is the Lorentz space norm defined as follows, as in  \cite[Prop. 1.4.9]{Grafakos}  \begin{equation*} 
 	\|f\|_{p,q} = \begin{cases}
 		\left( \frac{q}{p}  \int_0^\infty [f^*(\alpha) \alpha^{\frac{1}{p}} ]^q \frac{d\alpha}{\alpha}\right)^{\frac{1}{q}} \text{  when } q < \infty\\
 		\sup_{\alpha > 0} \,\alpha^{\frac{1}{p}} f^*(\alpha)  \hspace{1.14 cm}\text{ when } q =\infty.
 	\end{cases}
 \end{equation*}
  where $$ f^*({t})  =\inf \{ \alpha>0  : d_f(\alpha) \leq t \} $$ is called the \textit{non-increasing rearrangement} of $ f $.

It is known that for $p \in [1, \infty]$, $L^{p,p}(X)$ coincides with $L^p(X)$. Additionally,  the space $L^{p, \infty}(X)$ is referred to as the weak-$L^p(X)$ space.
We need the following properties of non-increasing rearrangements of functions: Let $f_1$ and $f_2$ are measurable functions on $(X, \mu)$. Then  
\begin{equation}
\label{properties-decreasing rearrangement-1}
(f_1f_2)^{*}(t_1+t_2) \leq f_1^{*}(t_1)f_2^{*}(t_2) \quad \text{ for } 0\leq t_1, t_2<\infty.
\end{equation}
and 
\begin{equation}\label{properties-decreasing rearrangement-2}
 \int_X |f|^p\,d\mu=\int_0^\infty f^*(t)^p\,dt\quad  \text{ for } 0<p<\infty.   
\end{equation}
\begin{equation}\label{properties-decreasing rearrangement-3}
    \left(\int_0^{\infty}\left(\frac{1}{(1/f_2)^*(t)}f_1^*(t)\right)^pdt\right)^{\frac{1}{p}} \leq C\left(\int_{X}{\left| f_1(x)f_2(x)\right|}^p d\mu\right)^{\frac{1}{p}}.
\end{equation}
The last inequality follows from the integral analogue of \cite[Theorem 368]{Ha} (see also \cite{He_84}).

\section{Pitt's inequality on symmetric spaces of noncompact type}\label{sec_Pitt_on_X}
This section is dedicated to proving Pitt's inequality with general weights, specifically Theorem \ref{thm_Pitt_nr_intn}. Before we delve into the proof, we need to recall a few results by Bradley\cite{Br}, Calderon \cite{Cal}, and Muckenhoupt \cite{Mu_72}.

Let $\sigma$ denote the closed segment in the square $0 \leq \alpha \leq 1$ and $0 \leq \beta \leq 1$ with endpoints $(\alpha_1, \beta_1)$ and $(\alpha_2, \beta_2)$, where $\alpha_1 \neq \alpha_2$ and $\beta_1 \neq \beta_2$. For such a segment, we associate two functions on $\R^+ \times \R^+$, namely
\begin{align*} 
   \psi(t,s):=\text{min}(s^{\alpha_1}/t^{\beta_1}, s^{\alpha_2}/t^{\beta_2}) \quad \text{ and } \quad \phi(t,s):=s\frac{d}{ds}\psi(t,s).
\end{align*}
A simple calculation shows that for $\alpha_1>0$, $\phi(t,s) \leq s^{\alpha_1}/t^{\beta_1}$. When  $\alpha_1 =0,$ then $ \alpha_2>0$, and 
\begin{align*}
            \phi(t,s) = \begin{cases}
                \alpha_2 {s^{\alpha_2}}/{t^{\beta_2}} \quad &\text{if } s\leq t^{\frac{\beta_2-\beta_1}{\alpha_2}},\\
                 0 \quad  & \text{if } s> t^{\frac{\beta_2-\beta_1}{\alpha_2}}.
            \end{cases}
    \end{align*}
    For $0<\alpha_1<\alpha_2$, we have
    \begin{align}\label{phi_exp}
        \phi(t,s) = \begin{cases}
                \alpha_2 {s^{\alpha_2}}/{t^{\beta_2}} \quad &\text{if } s \leq t^{m},\\
                   \alpha_1 {s^{\alpha_1}}/{t^{\beta_1}} \quad  &\text{if } s >t^{m},
                   \end{cases}
    \end{align}
    where $m= (\beta_2 -\beta_1)/(\alpha_2-\alpha_1)$.

We define an operator $S(\sigma)$ on functions on $\R^+$ defined by
\begin{align}\label{def_Ssig}
S(\sigma)f= \int_0^{\infty}\phi(t,s)f(s)\frac{ds}{s}.    
\end{align}
The following result is due to Calderon  \cite[Theorem 8]{Cal}: Let $\M_1$ and $\M_2$ are {\textit{sigma}} finite measure spaces.
\begin{theorem} \label{th:1.10}
Let $T$ be a sublinear operator defined on $L^{p_1,1}(\M_1)+L^{p_2,1}(\M_2)$ with measurable functions on $\M_2$ as values. Suppose that $T$ is simultaneously weak-types $(p_1,q_1)$ and $(p_2,q_2)$. Let $\sigma$ be the segment with $(1/p_1,{1}/{q_1}),~(1/p_2,{1}/{q_2})$ as endpoints. Then, there exists a constant $C>0$ such that
    \bes
    (Tf)^*(t) \leq CS(\sigma)f^*(t),
    \ees
    for all $t \in (0,\infty)$.
\end{theorem}
We need the following weighted Hardy inequalities. 
\begin{theorem}[{{\cite{Mu_72,Br}}}]\label{thm_brad}
Let $1\leq p \leq q \leq \infty$. Suppose $u, v$ and $F$ are non-negative functions defined on $(0,\infty)$. Then the following are true
\begin{itemize}
    \item[(i)] We have  \bes
\left(\int_0^{\infty}\left[u(t)\int_0^t F(s)ds\right]^q \,dt\right)^{\frac{1}{q}} \leq C\left(\int_0^{\infty}\left[v(t)F(t)\right]^p \, dt\right)^{\frac{1}{p}},
\ees
if and only if,
\bes
\sup_{s>0}\left(\int_s^{\infty}u(t)^qdt\right)^{\frac{1}{q}}\left(\int_0^sv(t)^{-p'}dt\right)^{\frac{1}{p'}}< \infty
\ees
with the usual modification if $p=1$ and $q=\infty$.
\item[(ii)]  Similarly, for the dual operator, we have
\bes
\left(\int_0^{\infty}\left[u(t)\int_t^{\infty}F(s)ds\right]^qdt\right)^{\frac{1}{q}} \leq C\left(\int_0^{\infty}\left[v(t)F(t)\right]^pdt\right)^{\frac{1}{p}},
\ees
if and only if,
\bes
\sup_{s>0}\left(\int_0^su(t)^qdt\right)^{\frac{1}{q}}\left(\int_s^{\infty}v(t)^{-p'}dt\right)^{\frac{1}{p'}}
<\infty.
\ees
\end{itemize}
\end{theorem}

We are now prepared to prove weighted Fourier inequalities on noncompact symmetric spaces $\X$ of rank one. The Hausdorff-Young inequality and the restriction theory developed by Ray and Sarkar \cite{RS_09} will play a fundamental role in our proof.

\begin{proof}[\textbf{Proof of Theorem \ref{thm_Pitt_nr_intn}}]
We first consider the case when $1<p\leq q<\infty$. Let us consider the following measure spaces $(\X,dx)$ and $(\R, |c(\lambda)|^{-2}\,d\lambda)$. For a given fixed $q_0 \in(1,\infty)$, we define a sublinear operator $\mathcal{T}_{q_0}$ on $C_c^\infty(\mathbb X)$ by 
\bes
\mathcal{T}_{q_0} f(\lambda) = \|\widetilde{f}(\lambda+i\rho_{q_0}, \cdot)\|_{L^{q_0}(K)}.
\ees
Then, by using the restriction inequality Theorem \ref{RS_4.2}, we have
\begin{align}\label{T_q0_1,inf}
\left( \int_K |\widetilde{f} (\lambda+i \rho_{q_0},k)|^{q_0} dk \right)^{\frac{1}{q_0}} \leq C_{q_0} \|f\|_{L^1(\X)},
\end{align}
for all $\lambda \in \R$. That is the operator $\mathcal{T}_{q_0}$ is of strong type $(1, \infty)$.  Next, we let $q_0 \in (1,2]$. By substituting $p=q=q_0$ in (the Hausdorff-Young inequality on symmetric spaces) Theorem \ref{thm_RS_HYinq}, we have
\begin{align*}
    \left( \int_{\R} \left( \int_K \left| \widetilde {f}(\lambda +i\rho_{q_0}, k)  \right|^{q_0} \, dk \right)^{\frac{q_0'}{q_0}} |c(\lambda)|^{-2} d\lambda \right) ^{\frac{1}{q_0'}} \leq C_{q_0} \|f\|_{L^{q_0}(\X)}.
\end{align*}
That is,
\begin{align}\label{T_q0_q_0q0'}
\|\mathcal{T}_{q_0} f\|_{L^{q_0'}(\R, \,|c(\lambda)|^{-2}d\lambda)}\leq C_{q_0}\|f\|_{L^{q_0}(\X)}.   
\end{align}
Thus, the operator $\mathcal{T}_{q_0}$ is of strong type $(1, \infty)$ and $(q_0, q_0')$. We recall that for any function $F$ on $\mathbb{R}$, $F^{\star}$ denotes the non-increasing rearrangement of $F$ with respect to the Plancherel measure $|c(\lambda)|^{-2} d\lambda$. Now, by applying Theorem \ref{th:1.10} to the operator $\mathcal{T}_{q_0}$, we get
\begin{equation}\label{T_q0_leq_phi}
    \begin{aligned}
    (\mathcal{T}_{q_0} f)^{\star}(t)& \leq C S(\sigma) (f^*)(t)\\
        & =  C  \int_0^{\infty} \phi(t,s) f^*(s)\, \frac{ds}{s},
\end{aligned}
\end{equation}
for all $ t\in (0,\infty)$, where $\sigma$ is the line segment with $({1}/{q}_0, {1}/{q_0'}), (1,0)$  as endpoints. For this segment $\sigma$, we express $\phi(t,s)$ as given in \eqref{phi_exp}:
\begin{align}\label{phi(t,s)exp_apply}
        \phi(t,s) = \begin{cases}   s \quad  & \text{if } s\leq t^{-1},
                \\
                 \\
                \frac{1}{q_0} \frac{s^{\frac{1}{q_0}}}{t^{\frac{1}{q_0'}}}  \quad &\text{if } s> t^{-1}.
                   \end{cases}
    \end{align}
    Plugging the formula above into the inequality \eqref{T_q0_leq_phi}, we obtain
    \begin{equation}\label{T_q0_<ex}
    \begin{aligned}
         (\mathcal{T}_{q_0} f)^{\star}(t)& \leq C \left(  \int_0^{t^{-1}} \phi(t,s) f^*(s)\, \frac{ds}{s}+   \int_{t^{-1}}^{\infty} \phi(t,s) f^*(s) \, \frac{ds}{s}\right)\\
         & = C \left(  \int_0^{t^{-1}}   f^*(s)\,  ds + \frac{1}{q_0} t^{-\frac{1}{q_0'}} \int_{t^{-1}}^{\infty} s^{-\frac{1}{q_0'}} f^*(s)\,  ds\right).  
    \end{aligned}
    \end{equation}
    By multiplying the inequality above by $U(t)$, integrating, and applying Minkowski's inequality, we can write for any $q\in (1,\infty)$
\begin{equation}\label{I1+I2_nr1}
\begin{aligned}
\left(\int_0^{\infty}\left(U(t)(\mathcal{T}_{q_0} f)^{\star}(t)\right)^q\, dt\right)^{\frac{1}{q}}&\leq C \left( I_1+I_2 \right),
 \end{aligned}
\end{equation}
where
\begin{align}
    I_1 &=  \left(\int_0^{\infty}U(t)^q\left(\int_0^{t^{-1}}f^*(s)ds\right)^qdt\right)^{\frac{1}{q}},\label{def_I1_nr1}\\
    I_2 & = \left(\int_0^{\infty}U(t)^q t^{-\frac{q}{q_0'}}\left(\int_{t^{-1}}^{\infty}s^{-\frac{1}{q_0'}}f^*(s)ds\right)^qdt\right)^{\frac{1}{q}}.\label{def_I2_nr1}
\end{align}
We analyze the estimates of $I_1$ and $I_2$ separately. The change of variable $t \mapsto t^{-1}$ in the integral inside \eqref{def_I1_nr1} implies that
\begin{align*}
     I_1 &=  \left(\int_0^{\infty}\left( {U(t^{-1})}{t^{-\frac{2}{q}}}\int_0^{t}f^*(s)ds\right)^qdt\right)^{\frac{1}{q}}.
\end{align*}
Here, we will utilize Hardy's inequalities.  Specifically, by applying Theorem \ref{thm_brad} (i) with $u(t) ={U(t^{-1})}{t^{-\frac{2}{q}}}$, $F(t)=f^*(t)$, and $1<p\leq q<\infty$, we derive
\begin{align}\label{est_I1_nr1}
    I_1 \leq  C\left(\int_0^{\infty}\left(V(t)f^*(t)\right)^pdt\right)^{\frac{1}{p}},
\end{align}
if and only if 
\begin{align*}
     \sup_{s>0}\left(\int_s^{\infty} {U(t^{-1})^q}{t^{-2}}dt\right)^{\frac{1}{q}}\left(\int_0^s V(t)^{-p'}dt\right)^{\frac{1}{p'}}< \infty,
\end{align*}
which, by a change of variable, transforms to the hypothesis \eqref{uv_loc_nr_int1}
\begin{align*}
   \sup_{s>0}\left(\int_{0}^{s^{-1}} U(t)^q dt\right)^{\frac{1}{q}}\left(\int_{0}^s V(t)^{-p'} dt\right)^{\frac{1}{p'}} < \infty.
\end{align*}
Next, to analyze $I_2$, we will approach similarly to the case of $I_1$. However, after the change of variable $t\rightarrow t^{-1}$ in \eqref{def_I2_nr1}, we employ Theorem \ref{thm_brad} (ii) with $u(t) = U(t^{-1}) t^{\frac{1}{q_0'}- \frac{2}{q}}$, $F(t)= t^{-\frac{1}{q_0'}}f^*(t)$, and $v(t)= t^{\frac{1}{q_0'}} V(t)$. This gives us that for any $1<p\leq q<\infty$
\begin{align}\label{est_I2_nr1}
    I_2 \leq  C\left(\int_0^{\infty}\left(V(t)f^*(t)\right)^p dt\right)^{\frac{1}{p}}
\end{align}
if and only if
\begin{align*}
    \sup_{s>0}\left(\int_0^s U(t^{-1})^q t^{\frac{q}{q_0'}-2} \, dt\right)^{\frac{1}{q}}\left(\int_s^{\infty} V(t)^{-p'}t^{-\frac{p'}{q_0'}} dt\right)^{\frac{1}{p'}} <\infty,
\end{align*}
which is equivalent to our hypothesis
\begin{align}\label{UV_glo_pf}
    \sup_{s>0}\left(\int_{s^{-1}}^{\infty} U(t)^q t^{-\frac{q}{q_0'}}dt\right)^{\frac{1}{q}}\left(\int_s^{\infty} V(t)^{-p'}t^{-\frac{p'}{q_0'}} dt\right)^{\frac{1}{p'}} < \infty.
\end{align}
Thus, combining the inequalities \eqref{I1+I2_nr1}, \eqref{est_I1_nr1}, and \eqref{est_I2_nr1}, we deduce that for $q_0\in (1,2]$
\begin{align}\label{uTq0<Vf}
    \left(\int_0^{\infty}\left(U(t)(\mathcal{T}_{q_0} f)^{\star}(t)\right)^q\, dt\right)^{\frac{1}{q}} & \leq  C\left(\int_0^{\infty}\left(V(t)f^*(t)\right)^p dt\right)^{\frac{1}{p}} 
\end{align}
for all $f\in C_c^{\infty}(\X)$. Therefore, by utilizing the property (\ref{properties-decreasing rearrangement-2}) of the non-decreasing rearrangement of a function and the inequality above, we can write
\begin{align*}
    \left(\int_{\R}\left( \int_K\left|  \widetilde{f}(\lambda +i\rho_{q_0},k) \right|^{q_0} dk\right)^{\frac{q}{q_0}} u(\lambda)^q |{c(\lambda)}|^{-2}d\lambda\right)^{\frac{1}{q}} & = \left(\int_0^{\infty}\left(U(t)(\mathcal{T}_{q_0} f)^{\star}(t)\right)^{q}dt\right)^{\frac{1}{q}}.
\end{align*}
After applying \eqref{uTq0<Vf}, it follows that
\begin{align*}
    \left(\int_{\R}\left( \int_K\left|  \widetilde{f}(\lambda +i\rho_{q_0},k) \right|^{q_0} dk\right)^{\frac{q}{q_0}} u(\lambda)^q |{c(\lambda)}|^{-2}d\lambda\right)^{\frac{1}{q}} & \leq C\left(\int_0^{\infty}\left(V(t)f^*(t)\right)^p dt\right)^{\frac{1}{p}} \\
   & =C\left(\int_0^{\infty}\left(\frac{1}{(1/v)^*(t)}f^*(t)\right)^pdt\right)^{\frac{1}{p}} \\
   &\leq C\left(\int_{\X}{\left| f(x)v(x)\right|}^pdx\right)^{\frac{1}{p}}.
\end{align*}
where in the last step, we used  \eqref{properties-decreasing rearrangement-3}. This concludes our Theorem \ref{thm_Pitt_nr_intn} for $q_0 \in (1,2]$. 

For the case $q_0>2$, the proof is the same as above except to replace $q_0$ by $q_0'$. For the case $p=1$ or $q=\infty$ the theorem will also follow similarly with obvious modifications. This completes the proof of Theorem \ref{thm_Pitt_nr_intn}.
\end{proof}
We observe here that we did not obtain the weighted inequality on the lines $\Im\lambda=\pm\rho$ (that is for $q_0=1$ or $q_0=\infty$). This is because, in these cases $q_0=1$ or $q_0=\infty$, the operator $\mathcal{T}_{q_0}$ (defined in the proof above) is only known to be of strong type $(1, \infty)$. No other strong or weak type estimate of the operator is known yet. Consequently, we could not apply Calder{\'o}n's estimate, Theorem \ref{th:1.10}.

Next, our aim is to demonstrate that, under certain conditions on $p$ and  $q$, the hypothesis (\ref{uv_loc_nr_int1}) implies (\ref{uv_glo_nr_int1}). We will follow the proof in {{\cite[Proposition 2.6]{He_84}}}. However, to accommodate our need, we will generalize their result further so that it follows as a corollary of the lemma presented below.
\begin{lemma}\label{1em_imply_new}
Let $q_0\in (1,\infty)$ and $U$ and $V$ be two nonegative functions defined on $(0,\infty)$, such that $U$ is non-increasing and $V$ is non-decreasing. Suppose $N\geq 1$, $1\leq p, q \leq \infty$. If either $p<{q_0}$ or $q>{q_0'}$, then
    \begin{align}\label{UV_locimp}
        \sup_{s>0}\left(\int_{0}^{s^{-1}} U(t)^q t^{N-1} dt\right)^{\frac{1}{q}}\left(\int_{0}^s V(t)^{-p'} t^{N-1} dt\right)^{\frac{1}{p'}} < \infty,
    \end{align}
    implies 
    \begin{align}\label{UV_gloimp}
        \sup_{s>0}\left(\int_{s^{-1}}^{\infty} U(t)^q t^{-\frac{Nq}{q_0'}+N-1}dt\right)^{\frac{1}{q}}\left(\int_s^{\infty} V(t)^{-p'}t^{-\frac{Np'}{q_0'}+N-1} dt\right)^{\frac{1}{p'}} < \infty.
    \end{align}
\end{lemma}
\begin{proof}
    We first consider the case $1 < p < q_0$ and $q<\infty$. Since $U$ is non-increasing, we can write
    \begin{align*}
        \int_{0}^{s} U(y)^q s^{N-1} dy\geq U(s)^q\int_{0}^{s} s^{N-1} dy=  U(s)^q s^{N},
    \end{align*}
   which, together with \eqref{UV_locimp}, implies that there exists a constant $C>0$ such that
\begin{align}\label{U(s)<}
U(s)s^{\frac{N}{q}}\leq C \left( \int_0^{s^{-1}} V(y)^{-p'} y^{N-1}dy\right)^{-\frac{1}{p'}},    
\end{align}
for all $s>0$. Replacing the variable $s$ with $t$ in the inequality above and utilizing the same, we can write
\begin{align*}
    I_1(s)&:=\left(\int_{s^{-1}}^{\infty} U(t)^q t^{-\frac{Nq}{q_0'}+N-1}dt\right)^{\frac{1}{q}}\left(\int_s^{\infty} V(t)^{-p'}t^{-\frac{Np'}{q_0'}+N-1} dt\right)^{\frac{1}{p'}}\\
    &\leq C\left(\int_{s^{-1}}^{\infty} \left( \int_0^{t^{-1}} V(y)^{-p'} y^{N-1} dy\right)^{-\frac{q}{p'}} t^{-N-\frac{Nq}{q_0'}+N-1}dt \right)^{\frac{1}{q}}\left(\int_{s}^{\infty} V(t)^{-p'}t^{-\frac{Np'}{q_0'}+N-1} dt\right)^{\frac{1}{p'}}\\
    &  \leq C \left(\int_{s^{-1}}^{\infty} V(t^{-1})^{q}t^{\frac{Nq}{p'}-\frac{Nq}{q_0'}-1} dt\right)^{\frac{1}{q}}V(s)^{-1}\left(\int_{s}^{\infty}t^{-\frac{Np'}{q_0'}+N-1} dt\right)^{\frac{1}{p'}}
\end{align*}
for all $s>0$, where in the last step, we used the non-decreasing property of $V$. Additionally, since $t \mapsto V(t^{-1})$ is non-increasing and $p < q_0$, we obtain the following from the inequality above
\begin{align*}
    I_1(s)\leq C \, V(s)s^{\frac{N}{q_0'}-\frac{N}{p'}}V(s)^{-1}s^{-\frac{N}{q_0'}+\frac{N}{p'}}= C.
\end{align*}

Next, we suppose $q_0'< q<\infty$ and $p>1$. Similarly as in \eqref{U(s)<},  using the non-decreasing property of  $V$,   we can derive that
\begin{equation*}
V(t)^{-1}t^{\frac{N}{p'}}\leq C \left( \int_0^{t^{-1}} U(y)^q y^{N-1}\,dy\right)^{-\frac{1}{q}}
\end{equation*}
for all $t>0$. Substituting the inequality above in the expression of $I_1(s)$ and using the non-decreasing property of $U$, it follows
\begin{align*}
    I_1(s) & \leq  C\left(\int_{s^{-1}}^{\infty} U(t)^q t^{-\frac{Nq}{q_0'}+N-1}dt \right)^{\frac{1}{q}}\left(\int_{s}^{\infty}\left( \int_0^{t^{-1}} U(y)^q y^{N-1}\,dy\right)^{-\frac{p'}{q}}t^{-N-\frac{Np'}{q_0'}+N-1} dt\right)^{\frac{1}{p'}}\\
&\leq C U(s^{-1})\left(\int_{s^{-1}}^{\infty}t^{-\frac{Nq}{q_0'}+N-1} dt\right)^{\frac{1}{q}}\left(\int_{s}^{\infty} U(t^{-1})^{-p'}t^{\frac{Np'}{q}-\frac{Np'}{q_0'}-1} dt\right)^{\frac{1}{p'}}\\
& = CU(s^{-1})s^{\frac{N}{q_0'}-\frac{N}{q}}U(s^{-1})^{-1}s^{-\frac{N}{q_0'}+\frac{N}{q}} \\
&= C,
\end{align*}
for all $s>0$, where in the second last step we used $q>q_0'$ to ensure the convergence of the integral near $t =\infty$. 
The case for $p=1$ and $q=\infty$ follows similarly with obvious modifications.
\end{proof}

From the lemma above and Theorem \ref{thm_Pitt_nr_intn}, we have the following result on Pitt's inequality.
\begin{corollary}\label{cor_p<q0}
For a given $1 \leq p \leq  q \leq \infty$, suppose the functions $U$ and $V$ satisfy the following inequality:
    \begin{equation} \nonumber
\sup_{s>0}\left(\int_{0}^{s^{-1}} U(t)^q dt\right)^{\frac{1}{q}}\left(\int_{0}^s V(t)^{-p'} dt\right)^{\frac{1}{p'}} < \infty.
\end{equation} Let $1<q_0\leq 2$ such that either $p<q_0$ or $q>q_0'$.
Then there exists a constant $C>0$ such that for all $f \in C_c^{\infty}(\X)$, the following holds
\bes
\left(\int_{\R}\left( \int_K\left|  \widetilde{f}(\lambda +i\rho_{q_0},k) \right|^{q_0} dk\right)^{\frac{q}{q_0}} u(\lambda)^q |{c(\lambda)}|^{-2}\, d\lambda\right)^{\frac{1}{q}} \leq C\left(\int_{\X} |f(x)|^{p} v(x)^p \,  dx\right)^{\frac{1}{p}}.
\ees
Moreover, for $q_0>2$, if $p<q_0'$ or $q>q_0$, then the inequality above holds true.
\end{corollary}

\section{Necessary condition for Pitt's inequality}\label{sec_nec}
In this section, we establish the necessary conditions for arbitrary radial weights for Pitt's inequality to hold for Fourier transforms in ${\X}$. But before we delve into that, we would like to recall that in the Euclidean setting, the classical Fourier transform $\mathcal{F}(f)$ for a radial function $f$ also becomes radial, given by
\begin{align*}
\mathcal{F}(f)(x) = C_N \int_{\mathbb{R}^N} f(x) j_{N/2-1}(|\xi|x) \, d\xi,
\end{align*}
where $j_{N/2-1}$ is the normalized Bessel function. By choosing $|\xi x|$ smaller than the first zero of $j_{N/2-1}$, one can bound it from below by a constant, which in turn provides a lower bound for $|\mathcal{F}(f)|$. Then, by appropriately choosing a compactly supported function (see \cite[page 579]{He_84}), we can prove the necessary condition \eqref{nec_org_RN} in $\mathbb{R}^{N}$.

 As we have mentioned in the introduction, in the case of symmetric spaces representing the spherical function $\phi_{\lambda} $ in terms of Bessel functions may not be useful. Furthermore, expecting $\phi_{\lambda+i\rho_p}$ to be bounded from below by a constant is unrealistic, as it has exponential decay towards infinity (see \eqref{est_phi_lam}).  Instead of a constant, we should expect to be bounded from below by $\phi_{i\rho_p}$. To address this, we utilized the integral representation (\ref{int_exp_phi}) of $\phi_{\lambda}$. Additionally, we have demonstrated that similar to the case for $ p\not=2$ where $\phi_{\lambda+i\rho_p}(a_t)$ is comparable to $\phi_{i\rho_p}(a_t)$,  in the degenerate case $p=2 $ (i.e. $\rho_p=0$), $\phi_{\lambda}(a_t)$ is comparable to $\phi_{0}(a_t)$ for  $\lambda \in \R$  provided $|\lambda t|$ is not too large.

\begin{lemma}\label{lem_f^geq}
    Let  $p\in [1,\infty]$, $\alpha>0$, and $B_{\alpha}=\{ x \in \X : |x| \leq \alpha\}$. Then for any  $K$-biinvariant non-negative function $f$ supported on $B_{\alpha}$, the following inequality holds for all $\lambda \in \R$
    \begin{equation}\label{f^geq}
    \begin{aligned}
        &|\what{f}(\lambda+i\rho_p)| \geq \int_0^{\alpha} f(a_t)\left( \frac{c_0 \sinh 2t}{\Delta(t)} \int_0^t (\cosh 2t- \cosh 2s)^{\frac{m_2}{2}-1} \right. \\
        & \hspace{2.7cm} \cdot \left. \left(\int_0^s (\cosh s -\cosh r)^{\frac{m_1}{2}-1} \cos(\lambda r) \cosh(\rho_p r) \,dr\right) \sinh s\, ds\right) \Delta(t)\, dt, 
    \end{aligned}
    \end{equation}
    for all $\lambda\in \R$. Moreover,  for any fixed positive constant $\theta_0<\pi/2$, if  $|\lambda \alpha|\leq \theta_0$,  then we have
    \begin{align}\label{f^>_phi}
        |\what{f}(\lambda+i\rho_p)|  \asymp_{\theta_0} \int_0^{\alpha} f(a_t) \phi_{i\rho_p}(a_t) \Delta(t) dt.
    \end{align}
\end{lemma}
\begin{proof}
We will demonstrate \eqref{f^geq} specifically for $\lambda \geq 0$, with the analogous reasoning applying to the other case. We recall $\phi_{\lambda}(x)=\phi_{\lambda}(x^{-1})$, for all $x\in G$. Using this, employing the Cartan integration formula \eqref{car_dec}, and  from \eqref{int_rep_phi}, we obtain
\begin{align*}
     \what{f}(\lambda+i\rho_p) &= \int_{G} f(x) \phi_{\lambda+ i\rho_p}(x^{-1}) \,dx\\
    &=\int_0^{\alpha} f(a_t) \phi_{\lambda+i\rho_p}(a_t) \Delta(t) \,dt\\
    &=\int_0^{\alpha} f(a_t)\left( \frac{c_0 \sinh 2t}{\Delta(t)} \int_0^t (\cosh 2t- \cosh 2s)^{\frac{m_2}{2}-1} \right. \\
        & \hspace{2.5cm} \cdot \left. \left(\int_0^s (\cosh s -\cosh r)^{\frac{m_1}{2}-1} \cos(\lambda r+i\rho_p r) \,dr\right) \sinh s\, ds\right) \Delta(t)\, dt.
\end{align*}
Substituting the cosine formula $$\cos (\lambda r + i \rho_p r) = \cos(\lambda r) \cosh(\rho_p r)- i \sin(\lambda r) \sinh(\rho_p r),$$
we can rewrite the expression as follows
\begin{align}\label{f^_deco}
    \what{f}(\lambda+i\rho_p) &= I_{\cos}(\lambda)- i I_{\sin}(\lambda),
\end{align}
where 
\begin{equation}
\begin{aligned}\label{defn_Icos}
     I_{\cos}(\lambda)& = \int_0^{\alpha} f(a_t)\left( \frac{c_0 \sinh 2t}{\Delta(t)} \int_0^t (\cosh 2t- \cosh 2s)^{\frac{m_2}{2}-1} \right. \\
        & \hspace{1cm} \cdot \left. \left(\int_0^s (\cosh s -\cosh r)^{\frac{m_1}{2}-1}  \cos(\lambda r)\cosh(\rho_p r)     \,dr\right) \sinh s\, ds\right) \Delta(t) \,dt,
\end{aligned}
\end{equation}
\begin{align*}
     \hspace{-1.1cm} I_{\sin}(\lambda)& = \int_0^{\alpha} f(a_t)\left( \frac{c_0 \sinh 2t}{\Delta(t)} \int_0^t (\cosh 2t- \cosh 2s)^{\frac{m_2}{2}-1} \right. \\
        & \hspace{1.1
        cm} \cdot \left. \left(\int_0^s (\cosh s -\cosh r)^{\frac{m_1}{2}-1}  \sin(\lambda r)\sinh(\rho_p r)     \,dr\right) \sinh s\, ds\right) \Delta(t) \, dt.
\end{align*}
We observe that both $I_{\cos}$ and $I_{\sin}$ are real-valued functions when $\lambda \in \R$. Therefore, taking the modulus on both sides of \eqref{f^_deco}, we arrive at
\begin{align}\label{f^>Icos}
| \what{f}(\lambda+i\rho_p)| \geq I_{\cos}(\lambda),
\end{align}
which is \eqref{f^geq}. If we assume $\lambda \alpha\leq \theta_0$, then by considering the expression \eqref{defn_Icos} of  $I_{\cos}(\lambda)$, we find that $r\leq \alpha$ implies $\cos (\lambda r)\geq \cos (\lambda \alpha)\geq C_{\theta_0}$. Substituting this into \eqref{f^>Icos}, we obtain
\begin{align*}
    | \what{f}(\lambda+i\rho_p)| &\geq C_{\theta_0} \int_0^{\alpha} f(a_t)\left( \frac{c_0 \sinh 2t}{\Delta(t)} \int_0^t (\cosh 2t- \cosh 2s)^{\frac{m_2}{2}-1} \right. \\
        & \hspace{3cm} \cdot \left. \left(\int_0^s (\cosh s -\cosh r)^{\frac{m_1}{2}-1}   \cosh(\rho_p r)     \,dr\right) \sinh s\, ds\right) \Delta(t) dt
\end{align*}
Now using the formula of $\phi_{i\rho_p}$ from \eqref{int_rep_phi}, we have from the inequality above
\begin{align}\label{what_f>}
    | \what{f}(\lambda+i\rho_p)|\geq   C_{\theta_0} \int_0^{\alpha} f(a_t) \phi_{i\rho_p}(a_t) \Delta(t) dt.
\end{align}
The other inequality follows easily from the property \eqref{phi_il} of $ \phi_{\lambda}$. Indeed, using \eqref{phi_il}, we can write for all $\lambda \in \R$,
\begin{align*}
    |\what{f}(\lambda +i\rho_p)|  \leq \int_{\X} |f(x)| |\phi_{\lambda+i\rho_p}(x^{-1})| \, dx
     \leq \int_{\X} f(x)\, \phi_{i \rho_p}(x) \, dx.
\end{align*}
From  \eqref{what_f>} and the inequality above, we conclude our lemma.
\end{proof}

For $\Im \lambda<0$, it follows from \eqref{est_phi_lam} that $|\phi_{\lambda}(a_t)|\asymp \phi_{i\Im \lambda}(a_t) $. However, for the case $\Im \lambda=0$, that is, $\lambda \in \R$,  it is not known whether this behavior holds. In the following, we investigate this degenerate case.
\begin{proposition}\label{sph_real_low}
 Suppose $0<\theta_0<\pi/2$ is a constant. Then, the following holds for $\lambda \in \R$
    \begin{align*}
        \phi_\lambda(a_t) \asymp_{\theta_0}  \phi_0 (a_t) \quad \text{whenever } 0\leq |\lambda t| \leq \theta_0.
    \end{align*}
\end{proposition}
\begin{proof}
      We recall the spherical function satisfies the following properties  
      \begin{align*}
          \phi_{\lambda}= \phi_{-\lambda} \quad \text{and} \quad \phi_{\lambda}(a_t)= \phi_{\lambda}(a_{-t}),
      \end{align*}
      for all $\lambda, t \in \R$. Thus, to prove the lemma, it is sufficient to consider the case where  $\lambda, t\geq 0$.
    We fix a $\theta_0 \in (0,\pi/2)$. Then,  for all $ s\leq \theta_0$, we have $\cos s \geq C_{\theta_0}$, where $C_{\theta}>0$ is constant depending only on $\theta_0$. Now, for $\lambda t\leq \theta_0$, we can leverage the formula \eqref{int_exp_phi} to write 
   \begin{equation*}
       \begin{aligned}
      {(c_0 \sinh 2t)}^{-1}{\Delta(t)} \phi_{\lambda}(a_t) &\\
      & \hspace{-3cm}= \int_0^t (\cosh 2t- \cosh 2s)^{\frac{m_2}{2}-1} \left(\int_0^s (\cosh s -\cosh r)^{\frac{m_1}{2}-1} \cos(\lambda r) \,dr\right) \sinh s\, ds\\
      & \hspace{-3cm}\geq  \int_0^t (\cosh 2t- \cosh 2s)^{\frac{m_2}{2}-1} \left(\int_0^s (\cosh s -\cosh r)^{\frac{m_1}{2}-1} \cos(\lambda t) \,dr\right) \sinh s\, ds\\
      &  \hspace{-3cm}\geq C_{\theta_0} \int_0^t (\cosh 2t- \cosh 2s)^{\frac{m_2}{2}-1} \left(\int_0^s (\cosh s -\cosh r)^{\frac{m_1}{2}-1}   \,dr\right) \sinh s\, ds,
   \end{aligned}
   \end{equation*}
   where in the second step we used the fact that $r \leq t$, and so $\cos (\lambda r)\geq \cos (\lambda t)$. Thus, from the inequality above, we obtain
   \begin{align}\label{geq_phi_0}
       \phi_{\lambda}(a_t) \geq C_{\theta_0} \phi_0(a_t), \quad \text{ whenever } 0\leq  \lambda t \leq \theta_0.
   \end{align}
   To prove the reverse inequality, we have from \eqref{phi_il} 
   \begin{equation}\label{leq_phi_0}
       |\phi_{\lambda}(a_t)|\leq \phi_0(a_t)
   \end{equation}
   for all $\lambda \in \R.$ Therefore from \eqref{geq_phi_0} and \eqref{leq_phi_0}, our lemma follows.
\end{proof}

We will now proceed with the proof of Theorem \ref{thm_nec_int}. Additionally, we will demonstrate a more general necessary condition from which Theorem \ref{thm_nec_int} will follow as a consequence. As in the Euclidean case, we assume that $u$ and $v^{-p'}$ are locally integrable. 

\begin{theorem}\label{thm_nec_new}
  Let {${q_0 \in [1,\infty]}$} be fixed.  Assume that for any two radial weights $u$ and $v$,  the Pitt's inequality \eqref{pitt_nrad_int} holds for any $1<p,q<\infty$. Then for any $\kappa\in \R$,
\begin{multline}\label{nec-new}
    \sup_{s>0} \left(\int_{0}^{\theta_0/s} u(\lambda)^q |c(\lambda)|^{-2}d\lambda\right)^{\frac{1}{q}} \left(\int_{0}^{s}v(a_t)^{-p'}  \phi_{i\rho_{q_0}}(a_t)^{\kappa p}  \Delta(t) dt\right)^{-\frac{1}{p}} \\  \cdot \left( \int_0^{s} v(a_t)^{-p'}\phi_{i\rho_{q_0}}(a_t)^{\kappa+1} \Delta(t) dt \right)< \infty,
\end{multline}
where $\theta_0$ is any fixed positive number less than $\pi/2$.
\end{theorem}
\begin{proof}
Let $u$ and $v$ be radial weights on $\R$ and $\X$ respectively. Moreover, assume that for the weights $u$ and $v$, Pitt's inequality \eqref{pitt_nrad_int} holds. Then, by employing the integration formula \eqref{car_dec} corresponding to the Cartan decomposition, we can express \eqref{pitt_nrad_int} as
    \be \label{pitt_rad_strip}
C\left(\int_{0}^{\infty}|v(a_t)f(a_t)|^p \Delta(t) dt\right)^{\frac{1}{p}} \geq \left(\int_{\R}\abs {u(\lambda)\hat{f}(\lambda + i \rho_{q_0})}^q\abs{c(\lambda)}^{-2}d\lambda\right)^{\frac{1}{q}} 
\ee
for all $f \in C_c^{\infty}(G//K)$. Here, we have utilized the fact that $\tilde{f}(\lambda,k) = \what f(\lambda)$ for all $\lambda \in \C, k\in K$. Now, for any $s>0$ and $\alpha \in \R$, let us consider the following radial function
\begin{align*}
    f_{\kappa}(a_t)=\begin{cases}
        v(a_t)^{-p'}\phi_{i\rho_{q_0}} (a_t)^{\kappa}   \quad &\text{if } |t|\leq s,\\
        0 \quad &\text{otherwise}.
    \end{cases} 
\end{align*}
Then, we obtain
\begin{equation}
\begin{aligned}\label{lp_fk}
    \left(\int_{0}^{\infty}|v(a_t)f_{\kappa}(a_t)|^p \Delta(t) dt\right)^{\frac{1}{p}}& = \left(\int_{0}^{s}v(a_t)^{p(1-p') } \phi_{i\rho_{q_0}} (a_t)^{\kappa p} \Delta(t) dt\right)^{\frac{1}{p}} \\
    & = \left(\int_{0}^{s}v(a_t)^{-p' } \phi_{i\rho_{q_0}} (a_t)^{\kappa p} \Delta(t) dt\right)^{\frac{1}{p}}.\end{aligned}
\end{equation}
On the other hand, we can write
\begin{align*}
    \left(\int_{\R} u(\lambda)^q |\what{f}_{\kappa}(\lambda+i\rho_{q_0})|^q |c(\lambda)|^{-2}  d\lambda\right)^{\frac{1}{q}} \geq \left(\int_{0}^{\theta_0 /s} u(\lambda)^q |\what{f}_{\kappa}(\lambda+i\rho_{q_0})|^q |c(\lambda)|^{-2}  d\lambda\right)^{\frac{1}{q}},
\end{align*}
for any $\theta_0>0$. We observe that in the right-hand side integral  $0\leq \lambda\leq \theta_0/s$. Also, we recall that $f_{\kappa}$ is a $K$-biinvariant function supported on the set $B_s=\{ x \in \X : |x| \leq s\}$. Thus, selecting $\theta_0$ as any fixed positive number less than $\pi/2$, and applying Lemma \ref{lem_f^geq}, we obtain 
\begin{align*}
    & \left(\int_{\R} u(\lambda)^q |\what{f}_{\kappa}(\lambda+i\rho_{q_0})|^q |c(\lambda)|^{-2}  d\lambda\right)^{\frac{1}{q}} \\ & 
 \hspace{2cm} \geq C_{\theta_0} \left(\int_{0}^{\theta_0 /s} u(\lambda)^q \left| \int_0^{s} f_{\kappa}(a_t) \phi_{ i \rho_{q_0}}(a_t) \Delta(t) dt \right|^q |c(\lambda)|^{-2}  d\lambda\right)^{\frac{1}{q}}\\
 & \hspace{2cm} =C_{\theta_0} \left(\int_{0}^{\theta_0 /s} u(\lambda)^q \left| \int_0^{s} v(a_t)^{-p'} \phi_{ i \rho_{q_0}}(a_t)^{\kappa+1} \Delta(t) dt \right|^q |c(\lambda)|^{-2}  d\lambda\right)^{\frac{1}{q}}\\
 & \hspace{2cm}= C_{\theta_0} \left( \int_0^{s} v(a_t)^{-p'} \phi_{ i \rho_{q_0}}(a_t)^{\kappa+1} \Delta(t) dt \right) \left(\int_{0}^{\theta_0 /s} u(\lambda)^q  |c(\lambda)|^{-2}  d\lambda\right)^{\frac{1}{q}}
\end{align*}
Combining the inequality above, with \eqref{lp_fk} and \eqref{pitt_rad_strip}, we get
\begin{multline*}
  C_{\theta_0}  \left(\int_{0}^{s}v(a_t)^{-p' } \phi_{i\rho_{q_0}} (a_t)^{\kappa p} \Delta(t) dt\right)^{\frac{1}{p}} \\ \geq  \left( \int_0^{s} v(a_t)^{-p'} \phi_{ i \rho_{q_0}}(a_t)^{\kappa+1} \Delta(t) dt \right) \left(\int_{0}^{\theta_0 /s} u(\lambda)^q  |c(\lambda)|^{-2}  d\lambda\right)^{\frac{1}{q}}
\end{multline*}
or,  equivalently, 
\begin{multline*}
     \left(\int_{0}^{\theta_0/s} u(\lambda)^q |c(\lambda)|^{-2}d\lambda\right)^{\frac{1}{q}} \left(\int_{0}^{s}v(a_t)^{-p' } \phi_{i\rho_{q_0}} (a_t)^{\kappa p} \Delta(t) dt\right)^{-\frac{1}{p}} \\  \cdot \left( \int_0^{s} v(a_t)^{-p'} \phi_{ i \rho_{q_0}}(a_t)^{\kappa+1} \Delta(t) dt \right)<C_{\theta_0} 
\end{multline*}
for all $s>0$. This concludes the proof of \eqref{nec-new}.
\end{proof} 
 
 \begin{proof}[\textbf{Proof of Theorem \ref{thm_nec_int}}]
    By  choosing $\kappa= \frac{1}{(p-1)}$  in the Theorem \ref{thm_nec_new} above, \eqref{nec-new} reduced to 
\begin{align*}
    \sup_{s>0} \left(\int_{0}^{\theta_0/s} u(\lambda)^q |c(\lambda)|^{-2}d\lambda\right)^{\frac{1}{q}} \left(\int_{0}^{s}v(a_t)^{-p'}  \phi_{i\rho_{q_0}}(a_t)^{p'}  \Delta(t) dt\right)^{\frac{1}{p'}}< \infty.
\end{align*}
This proves  Theorem \ref{thm_nec_int}.
 \end{proof}

\section{Pitt's inequality for polynomial and exponential weights}\label{sec_example}
This section is dedicated to the applications of our main results on weighted Fourier inequalities.
We provide a comprehensive method to analyze both sufficient and necessary conditions for exponential and polynomial weights to satisfy \eqref{uv_loc_int} using Theorem \ref{thm_Pitt_nr_intn} and Theorem \ref{thm_nec_int}.  To address the exponential volume growth, we decompose these conditions into two parts: local and global. We will observe that in the analysis of the local part, the dimension of the symmetric space plays a crucial role. However, the situation is entirely different for the global part, where the exponential volume growth of symmetric spaces dominates. We will begin by estimating the non-increasing rearrangement for polynomial and exponential weights.
\subsection{Example of weights and their non-increasing rearrangement}
\begin{enumerate}
   \item\label{exm-triv} \textit{(Trivial weight)}  We denote by $w_0$ the weight that is identically $1$. It is easy to verify that the non-increasing rearrangement with respect to an infinite measure is the constant function $1$.
    \item  \textit{(Weights on the function side)} For $\kappa, \delta\geq 0$, let us consider the following weight $$v=v_{\kappa, \delta}:=v_\kappa \omega_{\delta},$$ where 
    \begin{equation*}
    \begin{aligned}
 v_\kappa(a_s)&:=s^\kappa, \qquad \text{and} \qquad 
 \omega_{\delta}(a_s)&:= e^{2\rho \delta s}, 
 \end{aligned}
\end{equation*} for all $s\geq 0$, extended as a $K$-biinvaraint function on $G$ as in \eqref{A+_to_Kbi}. We first find the non-increasing rearrangement of the function $1/v_{\kappa}$ for $ \kappa>0$. The $\kappa=0$ case follows from the previous example \eqref{exm-triv}. Let $\kappa>0$ and $\mu$ be the Haar measure on $G$ in the Polar decomposition \eqref{car_dec}. Then the distribution function of $1/v_\kappa$ is 
\begin{align*}
    d_{1/v_{\kappa}}(\alpha)&=\mu \left(\left\{s \geq 0: v_{\kappa}^{-1}(s) >\alpha \right\}\right)\\
    &=  \mu \left(\left\{s \geq 0: s^{-\kappa} >\alpha \right\}\right)\\&= \mu \left( \left\{s:s <(1/\alpha)^{\frac{1}{\kappa}} \right\}\right).
\end{align*}
For $1 < \alpha < \infty$, using 
\begin{align}\label{est_Delta}
 \Delta(s) \asymp \begin{cases}
                s^{n-1}  \quad &\text{if } 0\leq s\leq 1,\\
                   e^{2\rho s} \quad  & \text{if } 1< s <\infty,
                   \end{cases}
\end{align}
 we get
\begin{align*}
     d_{1/v_{\kappa}}(\alpha)\asymp \int_0^{(1/\alpha)^{\frac{1}{\kappa}}} s^{n-1}\, dt \asymp \alpha^{-\frac{n}{\kappa}}.
\end{align*}
When $0<\alpha\leq 1$, 
\begin{align*}
     d_{1/v_{\kappa}}(\alpha)\asymp \int_{0}^{1} s^{n-1}ds + \int_1^{\alpha^{-\frac{1}{\kappa}}}e^{2\rho s}ds
    \asymp e^{2\rho \alpha^{-\frac{1}{\kappa}}}.
\end{align*}

The non-increasing rearrangement of $(1/v_{\kappa})^{*}$ of $1/v_{\kappa}$ is given by
\begin{align*}
     {(1/v_{\kappa})}^{*}(t) & = \inf\{\alpha>0: d_{1/v_{\kappa}}(\alpha)\leq t\} \\
                   & =  \min\{ \inf\{0<\alpha\leq 1: d_{{1/v_{\kappa}}}(\alpha) \leq t\}, \, \inf\{1<\alpha < \infty: d_{{1/v_{\kappa}}}(\alpha) \leq t\}\}.
\end{align*}
Now we make use of the estimate of $d_{1/v_{\kappa}}(\alpha)$ for $\alpha$ near zero and away from zero to derive the following
\begin{align*}
    (1/v_{\kappa})^*(t) & \asymp \min\{ \inf\{0<\alpha\leq 1: e^{2\rho \alpha^{-\frac{1}{\kappa}}} \lesssim t\}, \, \inf\{1<\alpha < \infty: \alpha^{-\frac{n}{\kappa}} \lesssim t\}\}.
\end{align*}

This gives
\begin{align*}
     (1/v_{\kappa})^*(t) & \asymp  \min\left\{ \inf\left\{0<\alpha\leq 1:  \alpha \geq  \left(\log t \right)^{-\kappa} \right\}, \, \inf\left\{1<\alpha < \infty: \alpha \geq  t^{-\frac{\kappa}{n}} \right\}\right\}.
\end{align*}
Hence, by considering $t$ is near and away from zero, we obtain from the above
\begin{align}
     (1/v_{\kappa})^*(t) \asymp &\begin{cases}
         t^{-\frac{\kappa}{n}} &\text{if  } 0\leq t\leq e^{2\rho},\\
         \left(\log t\right)^{-\kappa} &\text{if }e^{2\rho} <t<\infty.
     \end{cases}
\end{align}
Using a similar technique, it is easy to see that the estimate of non-increasing rearrangement $(1/\omega_{\delta})^{*}$ of $1/\omega_\delta$ is given by 
\begin{align*}
    (1/\omega_{\delta})^{*}(t)\asymp &\begin{cases}
        { e^{-2\rho \delta}} &\text{if  } 0\leq t\leq e^{2\rho},\\
          t^{-\delta}  &\text{if }e^{2\rho}<t<\infty.
     \end{cases}
\end{align*}
Moreover, using the property of non-increasing rearrangement \eqref{properties-decreasing rearrangement-1}, we have
\begin{align}\label{est_v_k,d*}
  \left(\frac{1}{v_{\kappa, \delta}}\right)^{*}(t) \leq \left(\frac{1}{v_{\kappa}}\right)^*\left(\frac{t}{2}\right)\left(\frac{1}{\omega_{\delta}}\right)^*\left(\frac{t}{2}\right)\asymp &\begin{cases}
        t^{-\frac{k}{n}} {e^{-2\rho \delta}} &\text{if  } 0\leq t\leq 2e^{2\rho},\\
          t^{-\delta}\left(\log\left(\frac{t}{2}\right)\right)^{-k}  &\text{if } 2e^{2\rho}<t<\infty.
     \end{cases}
\end{align}

    \item  \textit{(Weights on the Fourier transform side)}
 For $\sigma \geq 0$, we consider $u=u_\sigma$ given by \begin{equation*}
u_\sigma(\lambda)=|\lambda|^{-\sigma}, \qquad \text{for all $\lambda \in \R$.}
\end{equation*} 
Now we find the non-increasing rearrangement $u_\sigma^\star$ of $u_\sigma$ with respect to the measure $d\nu(\lambda):=|c(\lambda)|^{-2} d\lambda$ for $\sigma> 0$. We recall the following estimate of $ |c(\lambda)|^{-2}$ from Lemma \ref{est_c-2}
 \begin{equation}\label{est_c-2_d}
 \begin{aligned}
        |c(\lambda)|^{-2} \asymp \begin{cases}
            |\lambda|^2 \qquad &\text{if }|\lambda|\leq 1\\
            |\lambda|^{n-1} \qquad &\text{if } |\lambda| > 1.
        \end{cases}
    \end{aligned}
    \end{equation}
The distribution function of $u_\sigma$ is given by
 \begin{align*}
    \hspace{1.2cm} d_{u_{\sigma}}(\alpha)=  \nu\{\lambda \in \R:   |\lambda|^{-\sigma} >\alpha\} = \nu\{\lambda \in \R:   |\lambda|^{\sigma} < 1/\alpha\} = \nu\{\lambda \in \R:   |\lambda| < \left( 1/\alpha\right)^{\frac{1}{\sigma}}\}.
 \end{align*}
 For $0< \alpha\leq 1 $, using  \eqref{est_c-2_d} we have
 \begin{align*}
     d_{u_{\sigma}}(\alpha)\asymp \int_{0}^{1} \lambda^{2} d\lambda+  \int_{1}^{\left( 1/\alpha\right)^{\frac{1}{\sigma}}} \lambda^{n-1} d\lambda \asymp {\alpha^{-\frac{n}{\sigma}}}.
 \end{align*}
When $1<\alpha<\infty$, we  write
\begin{align*}
      d_{u_{\sigma}}(\alpha) \asymp \int_{0}^{\left( 1/\alpha\right)^{\frac{1}{\sigma}}} \lambda^{2}  d\lambda \asymp \alpha^{-\frac{3}{\sigma}}.
\end{align*}
Thus, from the definition of $ u_{\sigma}^{\star}$, we have
\begin{align*}
    u_{\sigma}^{\star}(t) & = \inf\{\alpha>0: d_{u_{\sigma}}(\alpha)\leq t\} \\
                   &\asymp \min\{ \inf\{0<\alpha\leq 1: {\alpha^{-\frac{n}{\sigma}}}\lesssim t\}, \, \inf\{1< \alpha< \infty: \alpha^{-\frac{3}{\sigma}} \lesssim t\}\}\\
    & \asymp \min\{ \inf\{0<\alpha\leq 1: {\alpha \gtrsim  t^{-\frac{\sigma}{n}}}\}, \, \inf\{1< \alpha< \infty: \alpha \gtrsim t^{-\frac{\sigma}{3}}\}\}.
\end{align*}
Therefore,  we have obtained the following
\begin{align}\label{est_u_s^*}
     u_{\sigma}^{\star}(t)\asymp
 \begin{cases}
        t^{-\frac{\sigma}{3}} \qquad &\text{if } 0 \leq t\leq 1\\
        t^{-\frac{\sigma}{n}} \qquad & \text{if } 1<t< \infty.\\
        \end{cases}
\end{align}
\end{enumerate}
\subsection{Sufficient condition for polynomial and exponential weights}
 Now that we have obtained estimates for the non-increasing rearrangement of the weights, we will use Theorem \ref{thm_Pitt_nr_intn} to determine the sufficient conditions on the weights $u_{\sigma}$ and $v_{\kappa, \delta}$ for the weighted Fourier inequality \eqref{pitt_nrad_int} to hold.

Let us denote the non-increasing rearrangement of $u_{\sigma}$ with respect to the measure $|c(\lambda)|^{-2}d\lambda$ by $U$, and the non-increasing rearrangement of $1/v_{\kappa, \delta}$ by $1/V$. We will first find the conditions on $\sigma, \kappa, \delta$ so that the following holds for $ 1<p\leq q<\infty$
\begin{align*}
    \sup_{0< s<\infty} \left(\int_0^{s^{-1}}U(t)^qdt\right)^{\frac{1}{q}} \left(\int_0^sV(t)^{-p'}dt\right)^{\frac{1}{p'}}  <\infty.
\end{align*}
Since the estimates of $U(t)$ and $V(t)$ differ significantly when $t$ is near zero and when it is away from zero, it is natural to consider these two cases separately: when $s$ is near zero and when $s$ is away from zero. Let us first examine the case when $s$ is near zero.

 For $0<s \leq 2e^{2\rho} $ and $\sigma q<3$, we can write from  \eqref{est_u_s^*}
\begin{align*}
   \left(\int_0^{s^{-1}}U(t)^qdt\right)^{\frac{1}{q}} & \asymp \left(\int_0^1t^{-\frac{\sigma q}{3}}dt+\int_1^{s^{-1}}t^{-\frac{\sigma q}{n}}dt\right)^{\frac{1}{q}}\\
    & \asymp s^{\frac{\sigma}{n}-\frac{1}{q}}.
\end{align*}
 For the same range of $s$ as above  and $kp'<n$, we get from  \eqref{est_v_k,d*}
\begin{align*}
     \left(\int_0^sV(t)^{-p'}\,dt\right)^{\frac{1}{p'}} \leq C \left(\int_0^s t^{-\frac{kp'}{n}}\,dt\right)^{\frac{1}{p'}} = C s^{-\frac{k}{n}+\frac{1}{p'}}.
\end{align*}
Therefore, we have
\begin{align*}
   \sup_{0<s\leq 2e^{2\rho}} \left(\int_0^sV(t)^{-p'}dt\right)^{\frac{1}{p'}} \left(\int_0^{s^{-1}}U(t)^qdt\right)^{\frac{1}{q}} \leq C \sup_{0<s\leq 2e^{2\rho}} s^{-\frac{k}{n}+\frac{1}{p'}+\frac{\sigma}{n}-\frac{1}{q}}.
\end{align*}
The expression above is finite if $$\frac{\sigma}{n}  +\frac{1}{p'} \geq \frac{\kappa}{n}+ \frac{1}{q}. $$\\
Now let $2e^{2\rho}<s<\infty$. Then from \eqref{est_v_k,d*}
\begin{align*}
    \left(\int_0^s V(t)^{-p'}dt\right)^{\frac{1}{p'}} & \leq  C \left(\int_0^{2e^{2\rho}} t^{-\frac{kp'}{n}}e^{-2\rho \delta p'}dt+\int_{2e^{2\rho}}^s t^{-\delta p'}\left(\log\left(\frac{t}{2}\right)\right)^{-kp'}dt\right)^{\frac{1}{p'}}\\
   & \leq  C \left(\int_{2e^{2\rho}}^s t^{-\delta p'}\left(\log\left(\frac{t}{2}\right)\right)^{-kp'}dt\right)^{\frac{1}{p'}}\\
    & \leq C  \left(\int_{e^{2\rho}}^s t^{-\delta p'}dt\right)^{\frac{1}{p'}}\\
    & \leq C s^{-\delta +\frac{1}{p'}}.
\end{align*}
For the same range of $s$, we have
\begin{align*}
    \left(\int_0^{s^{-1}}U(t)^qdt\right)^{\frac{1}{q}} & \asymp \left(\int_0^{s^{-1}}t^{-\frac{\sigma q}{3}}dt\right)^{\frac{1}{q}}\\
     & \asymp s^{\frac{\sigma}{3}-\frac{1}{q}}.
\end{align*}
 Therefore, we have
\begin{align*}
   \sup_{2e^{2\rho}\leq s<\infty} \left(\int_0^sV(t)^{-p'}dt\right)^{\frac{1}{p'}} \left(\int_0^{s^{-1}}U(t)^qdt\right)^{\frac{1}{q}} \leq C \sup_{2e^{2\rho}\leq  s<\infty} s^{-\delta +\frac{1}{p'}+\frac{\sigma}{3}-\frac{1}{q}}
\end{align*}
which will be finite  if $$ \frac{\sigma}{3} +\frac{1}{p'} \leq \delta +\frac{1}{q}.$$
Similarly, we  can find that condition (\ref{uv_glo_nr_int1}) holds for $u = u_\sigma$ and $v = v_{\kappa, \delta}$, if the following additional conditions are satisfied for $1<q_0\leq 2$,
  \begin{equation*}
 \left(\frac{\sigma}{n} + \frac{1}{q_0'}\right)q>1,\,\,\, \left(\delta +\frac{1}{q_0'}\right)p'>1.    \end{equation*}

Theorem \ref{thm_Pitt_nr_intn}, Corollary \ref{cor_p<q0} and the calculations above lead us to the following result, which provides sufficient conditions on $u=u_\sigma$ and $v=v_{\kappa, \delta}$ for Pitt's inequality (\ref{pitt_nrad_int}) to hold.
\begin{corollary}
  Let $1  < p \leq q < \infty$. Then the inequality (\ref{pitt_nrad_int}) holds true for the weights $u=u_\sigma$ and $v=v_{\kappa, \delta}$, if the following inequalities are satisfied: 
\begin{enumerate}
    \item When $1<q_0\leq 2$ and either $p<q_0$ or $q>q_0'$:
\begin{equation*}
    \delta\geq 0, \,\,\,0\leq \kappa <\frac{n}{p'},  \,\,\,0\leq \sigma <\frac{3}{q}
\end{equation*}
and 
\begin{equation*} 
\kappa -\sigma\leq  n\left(1-\frac{1}{p}-\frac{1}{q} \right),\,\,\, \frac{\sigma}{3} +\frac{1}{p'} \leq \delta +\frac{1}{q}.
\end{equation*} 
\item In general, when $1<q_0\leq 2$:
 \begin{equation} \label{cond-v-kappa-delta-1}
0\leq \kappa <\frac{n}{p'},  \,\,\,0\leq \sigma <\frac{3}{q},\, \sigma> n\left(\frac{1}{q}-\frac{1}{q_0'} \right),\,\,\,\delta\geq 0, \,\,\, \delta> \left(\frac{1}{q_0}- \frac{1}{p}\right),
\end{equation}
and 
\begin{equation}\label{cond-v-kappa-delta-2} 
\kappa -\sigma\leq  n\left(\frac{1}{p'}-\frac{1}{q} \right),\,\,\, \frac{\sigma}{3} +\frac{1}{p'} \leq \delta +\frac{1}{q}. 
\end{equation} 
Moreover, for the case when $2<q_0<\infty$, the same conditions (as above) are sufficient, except $q_0$ to be replaced by $q_0'$.
\end{enumerate}

\end{corollary}
\begin{remark}
 \begin{enumerate}
 \item By taking $\kappa=\sigma =\delta=0$, fixing $p\in (1,2)$, and $q=p'$, while varying $q_0$ between $p$ and $p'$ in the corollary above, we recover the Hausdorff-Young inequality (see Theorem \ref{thm_RS_HYinq}), except at the boundary lines of the strip $S_p$.
 \item From the sufficient conditions (\ref{cond-v-kappa-delta-1}) and (\ref{cond-v-kappa-delta-2}), we do not find any nonzero values of $\sigma$ and $\kappa$ such that (\ref{pitt_rad_int}) holds for the (polynomial) weights $u=u_\sigma$ and $v=v_\kappa$, for $n>3$. However, if we replace $v_\kappa$ with the (polynomial times exponential) weight $v_{\kappa, \delta}$, where $\delta>0$, we can find sufficient $\kappa$ and $\sigma$ such that (\ref{pitt_nrad_int}) holds.
\item For $n=3$ and $q_0=2$, if $1< p<2$ and $p\leq q<\infty$ and $\sigma\geq 0$ satisfy \begin{equation*}
         \sigma <\frac{3}{q},\quad \sigma>3\left(\frac{1}{q}-\frac{1}{2} \right), \quad \frac{1}{p}+ \frac{1}{q}=1+\frac{\sigma}{3},
          \end{equation*} (that is, $\kappa=0, \delta=0$ in (\ref{cond-v-kappa-delta-1}) and (\ref{cond-v-kappa-delta-2}) then we have
     \bes
\left(\int_{\R} |\lambda|^{-\sigma q} \|\tilde{f}(\lambda,\cdot)\|_{L^2(K)}^q |{c(\lambda)}|^{-2}\, d\lambda\right)^{\frac{1}{q}} \leq C\left(\int_{\X}|{f(x)}|^p \, dx\right)^{\frac{1}{p}}. 
\ees
\item Let $n=2$ and $q_0=2 $. If $1< p<2$, $p\leq q<\infty$, and $\sigma, \kappa\geq 0$ satisfy\begin{equation*}
         \frac{2}{q}-1<\sigma <\frac{3}{q}, \quad \kappa <\frac{2}{p'}, \quad \kappa-\sigma\leq 2\left(1- \frac{1}{p}-\frac{1}{q}\right),\quad  \sigma\leq 3\left(\frac{1}{p}+\frac{1}{q}-1 \right) 
          \end{equation*}
          (that is, we put $\delta=0$ in \eqref{cond-v-kappa-delta-1} and \eqref{cond-v-kappa-delta-2}), then we have
     \bes
\left(\int_{\R} |\lambda|^{-\sigma q} \|\tilde{f}(\lambda,\cdot)\|_{L^2(K)}^q \abs{c(\lambda)}^{-2}d\lambda\right)^{\frac{1}{q}} \leq C\left(\int_{\X}|x|^{\kappa p}\abs{f(x)}^p\,dx\right)^{\frac{1}{p}}. 
\ees
 \end{enumerate}   
\end{remark}

\subsection{Necessary condition for polynomial and exponential weights}
In this section,  we will first determine the necessary conditions for the weights  $u=u_{\sigma}$  and $v=v_{\kappa, \delta}$ for the  Pitt's inequality \eqref{pitt_nrad_int} to hold.  Then, we will discuss several consequences of this result by considering particular cases.
\begin{theorem}

 Let $q_0 \in [1,2]$ and $n\geq 3$. Assume that for $u=u_{\sigma}$ and $v=v_{\kappa,\delta}$, the Pitt's inequality \eqref{pitt_nrad_int} holds for any  $1<p,q<\infty$. Then the following conditions are necessary:
    \begin{align}\label{nec_s,k,d}
         \sigma<\frac{3}{q}, \quad  \kappa< \frac{n}{p'}, \quad  \delta \geq \frac{1}{p'}- \frac{1}{q_0'},
    \end{align} 
    and 
    \begin{align}\label{eqn_balance_X}
        \kappa-\sigma \leq n \left(1-\frac{1}{p} -\frac{1}{q} \right). 
    \end{align}
    For the endpoint case $\delta = \frac{1}{p'}- \frac{1}{q_0'}$, the following conditions are also necessary
    \begin{enumerate}
    \item For $q_0 \not =2$ 
\begin{align*}
    \quad  \sigma+ 1-\frac{1}{p} \leq \kappa+\frac{3}{q}.
\end{align*}
     \item For $q_0=2$, 
     \begin{align*}
    \quad  \sigma +2-\frac{1}{p}\leq \kappa + \frac{3}{q}.
\end{align*}
 \end{enumerate}
Moreover, when $q_0 \in (2,\infty]$, the necessary conditions involving $q_0'$  will be replaced by  $q_0$.
 \end{theorem}
 \begin{proof}
Let $1<p,q<\infty$ be fixed.  From Theorem \ref{nec_org}, we recall the necessary condition  for $u= u_{\sigma}$ and $v={v_{\kappa,\delta}}$, 
 \begin{align}\label{nec_u_s_V_k,d}
     \sup\limits_{s>0}  \left(\int_{0}^{\theta_0/s} u_{\sigma}(\lambda)^q  |c(\lambda)|^{-2} d\lambda   \right)^{\frac{1}{q}} \left(   \int_0^s  \left(  \frac{v_{\kappa, \delta}(a_t)}{\phi_{i\rho_{q_0}}(a_t)} \right)^{-p'} \Delta(t) dt \right)^{\frac{1}{p'}}<\infty,
\end{align}
 where $\theta_0$ is any positive number  less than $\pi/2$. Let us fix a number $\theta_0 \in (0, \pi/2)$. For given $p$ and $q$, we define:
 \begin{align*}
     I^{\sigma}_{\kappa, \delta}(s)=  \left(\int_{0}^{\theta_0/s} u_{\sigma}(\lambda)^q |c(\lambda)|^{-2} \, d\lambda \right)^{\frac{1}{q}} \left(   \int_0^s  \left(  \frac{v_{\kappa, \delta}(a_t)}{\phi_{i\rho_{q_0}}(a_t)} \right)^{-p'} \Delta(t) dt \right)^{\frac{1}{p'}}.
 \end{align*}
To address the exponential volume growth of noncompact type symmetric spaces, we express
\begin{align}\label{I_dec}
    \sup\limits_{s>0}  I^{\sigma}_{\kappa, \delta}(s)= \max \left\{  I^{\sigma, \text{loc}}_{\kappa, \delta}, \, I^{\sigma,\text{glo}}_{\kappa, \delta}\right\},
\end{align}
 where 
\begin{align*}
    I^{\sigma, \text{loc}}_{\kappa, \delta}: =\sup\limits_{0< s\leq 1}  I^{\sigma}_{\kappa, \delta}(s)\qquad \text{and} \qquad  I^{\sigma, \text{glo}}_{\kappa, \delta}: =\sup\limits_{ 1<s<\infty}  I^{\sigma}_{\kappa, \delta}(s).
\end{align*}
It is clear from \eqref{I_dec} that it is sufficient to find necessary conditions on $\sigma$, $\kappa$, and $\delta$ to ensure that both $I^{\sigma, \text{loc}}_{\kappa, \delta}$ and $I^{\sigma, \text{glo}}_{\kappa, \delta}$ are finite. To achieve this, we will determine the lower bound of $I^{\sigma}_{\kappa, \delta}(s)$ for both small and large values of $s$. We will begin with the following integral 
\begin{align*}
\left(\int_{0}^{\theta_0/s} u_{\sigma}(\lambda)^q   |c(\lambda)|^{-2}\,  d\lambda \right)^{\frac{1}{q}}.
\end{align*}
For $0<s\leq 1$, we have from \eqref{est_c-2_d}
\begin{align}\label{int_u_s}
   \left(\int_{0}^{\theta_0/s} u_{\sigma}(\lambda)^q   |c(\lambda)|^{-2}\,  d\lambda \right)^{\frac{1}{q}}\asymp \left( \int_0^{\theta_0} \lambda^{-\sigma q +2}  \, d\lambda +\int_{\theta_0}^{\theta_0/s} \lambda^{-\sigma q +n-1}  \, d\lambda\right)^{\frac{1}{q}}.
\end{align}
We observe that the first integral on the right-hand side will be finite if and only if $\sigma q <3$. Assuming this condition, it follows that
 \begin{align}\label{est_u_s<1}
   \left(\int_{0}^{\theta_0/s} u_{\sigma}(\lambda)^q   |c(\lambda)|^{-2}\,  d\lambda \right)^{\frac{1}{q}}\asymp \left( 1+ {s}^{\sigma q-n}  \right)^{\frac{1}{q}} \asymp s^{\sigma -\frac{n}{q}}.
\end{align}
Next, we consider large values of $s$, specifically, $s \geq 1$. Following similar steps as above, we can write
\begin{align*}
      \left(\int_{0}^{\theta_0/s} u_{\sigma}(\lambda)^q   |c(\lambda)|^{-2}\,  d\lambda \right)^{\frac{1}{q}}\asymp  \left( \int_{0}^{\theta_0/s} \lambda^{-\sigma q +2}  \, d\lambda \right)^{\frac{1}{q}}.
\end{align*}
Once again, we observe that the above integral will converge near $\lambda=0$, if and only if $\sigma q <3$, as assumed in \eqref{est_u_s<1}. Thus, we have
\begin{align}\label{int_u_s>1}
      \left(\int_{0}^{\theta_0/s} u_{\sigma}(\lambda)^q   |c(\lambda)|^{-2}\,  d\lambda \right)^{\frac{1}{q}}\asymp    s^{\sigma -\frac{3}{q}}.
\end{align}
We will now determine the lower bound of the following integral
\begin{align*}
 \int_0^s  \left(  \frac{v_{\kappa, \delta}(a_t)}{\phi_{i\rho_{q_0}}(a_t)} \right)^{-p'} \Delta(t) dt.
\end{align*}
For $q_0\in [1,2)$, using the sharp estimate \eqref{est_phi_{rho_p}} of $\phi_{i\rho_{q_0}}$ for all $t\geq 0$ we have
\begin{align}\label{est_v_k,d_<2}
     \frac{v_{\kappa, \delta}(a_t)}{\phi_{i\rho_{q_0}}(a_t)} \asymp t^{\kappa}  e^{2\rho \left( \delta+ \frac{1}{q_0'}\right) t},
\end{align}
and if $q_0 \in (2,\infty]$
\begin{align*}
     \frac{v_{\kappa, \delta}(a_t)}{\phi_{i\rho_{q_0}}(a_t)} \asymp t^{\kappa}  e^{2\rho \left( \delta+ \frac{1}{q_0}\right) t}.
\end{align*}
When $q_0=2$, we have from \eqref{est_phi_0}
\begin{align*}
    \phi_0(a_t) \asymp (1+t) e^{-\rho t} ,\quad \text{for }t\geq 0,
\end{align*}
which leads us to the following 
\begin{align}\label{est_v_k,d_2}
     \frac{v_{\kappa, \delta}(a_t)}{\phi_{0}(a_t)} \asymp \frac{t^{\kappa}}{1+t}  e^{2\rho \left( \delta+ \frac{1}{2}\right) t}, \quad \text{for }t\geq 0.
\end{align}
Hence, for $0<s\leq 1$ and for $q_0\in [1,2]$, using  \eqref{est_Delta} the estimate of $\Delta(t)$ near  $t=0$, we can write 
\begin{align*}
 \int_0^s  \left(  \frac{v_{\kappa, \delta}(a_t)}{\phi_{i\rho_{q_0}}(a_t)} \right)^{-p'} \Delta(t) dt   \asymp 
  \int_0^s t^{-\kappa p'+n-1}.
\end{align*}
The integral above will be finite if and only if $\kappa p'<n$. Assuming so, we obtain for all $ 0<s\leq 1$
\begin{align}\label{int_v_k,d<1}
 \left(   \int_0^s  \left(  \frac{v_{\kappa, \delta}(a_t)}{\phi_{i\rho_{q_0}}(a_t)} \right)^{-p'} \Delta(t) dt \right)^{\frac{1}{p'}} \asymp s^{\frac{n}{p'} -\kappa}.
\end{align}
Combining the inequalities \eqref{est_u_s<1} and \eqref{int_v_k,d<1}, we can write assuming $\kappa<n/p'$ and $\sigma  <3/q$
\begin{align*}
   I^{\sigma, \text{loc}}_{\kappa, \delta} \asymp \sup\limits_{0\leq s<1} s^{\sigma -\frac{n}{q}+\frac{n}{p'} -\kappa}.
\end{align*}
To summarize the case $0<s\leq 1$, we can say the following conditions are necessary 
\begin{align}\label{nec_sum_s<1}
\sigma<\frac{3}{q}, \qquad \kappa < \frac{n}{p'}, \qquad \text{and} \qquad \sigma+ \frac{n}{p'} \geq \kappa + \frac{n}{q}.
\end{align}
so that \eqref{nec_u_s_V_k,d} is true.

We now consider the case $s>1$. Given the distinct behavior of the spherical functions $\phi_0(a_t)$  and $\phi_{i\rho_{q_0}}(a_t)$ for $q_0\not =2$ near $t=\infty$, as seen in  \eqref{est_phi_0} and \eqref{est_phi_{rho_p}} respectively, we need to treat the two cases $q_0\not =2$ and $q_0=2$ separately.

We will start by addressing the case $q_0 \in [1,2)$; when {$q_0\in (2,\infty]$}, the analysis is similar. From \eqref{est_v_k,d_<2}, we can express 
\begin{align}\label{est_s>1vk,d<2}
    \int_0^s  \left(  \frac{v_{\kappa, \delta}(a_t)}{\phi_{i\rho_{q_0}}(a_t)} \right)^{-p'} \Delta(t) \,  dt \asymp \int_0^1 t^{-\kappa p'+n-1}+ \int_1^{s} t^{-\kappa p'} e^{2\rho \left( 1- \beta p'\right) t} \, dt
\end{align}
where $\beta:= \delta+1/{q_0'} $. Similar to the previous case $s\leq 1$, we observe that the first integral on the right-hand side will not be finite unless $\kappa p' <n$. Additionally, we claim that $1-\beta p' \leq 0$ is a necessary condition for Pitt's inequality \eqref{pitt_nrad_int} to hold true with $u=u_{\sigma}$ and $v=v_{\kappa, \delta}$. If not, let us assume on the contrary  $\beta< 1/{p'}$. 
By noting that $t\mapsto t^{-\kappa p'}$ is a decreasing function, it follows from \eqref{est_s>1vk,d<2} that 
\begin{align*}
     \int_0^s  \left(  \frac{v_{\kappa, \delta}(a_t)}{\phi_{i\rho_{q_0}}(a_t)} \right)^{-p'} \Delta(t) \,  dt \geq C s^{-\kappa p'} e^{2\rho \left( 1- \beta p'\right) s}.
\end{align*}
Therefore, using the inequality above and \eqref{int_u_s>1}, we get
\begin{align*}
     I^{\sigma, \text{glo}}_{\kappa, \delta} \geq C   \sup\limits_{s>1} s^{\sigma -\frac{3}{q}-\kappa p'} e^{ 2\rho s \left( \frac{1}{p'}-\beta \right)}.
\end{align*}
Since we assumed $1/{p'}>\beta$, the right-hand side of the inequality above is infinite, which contradicts \eqref{nec_u_s_V_k,d}. Hence, we must have $\beta p' \geq 1$, which implies $t\mapsto e^{2\rho \left( 1- \beta p'\right) t} $ is a decreasing function and consequently
\begin{align*}
    \int_1^{s} t^{-\kappa p'} e^{2\rho \left( 1- \beta p'\right) t}\, dt   \geq e^{2\rho \left( 1- \beta p'\right) s} \int_1^{s} t^{-\kappa p'} \, dt \geq  C  e^{2\rho \left( 1- \beta p'\right) s} s^{1-\kappa p'}.    
\end{align*}
Hence, for $q_0\in [1,2)$ and for $s>1$, we have from (\ref{est_s>1vk,d<2})
\begin{align}\label{int_v_k,d>1}
  \left(  \int_0^s  \left(  \frac{v_{\kappa, \delta}(a_t)}{\phi_{i\rho_{q_0}}(a_t)} \right)^{-p'} \Delta(t) \,  dt \right)^{\frac{1}{p'}} \geq C  e^{2\rho \left( \frac{1}{p'}- \beta \right) s} s^{\frac{1}{p'}-\kappa }.
\end{align}
Thus, we have from \eqref{int_u_s>1} and \eqref{int_v_k,d>1}
\begin{align*}
     I^{\sigma, \text{glo}}_{\kappa, \delta} \geq C   \sup\limits_{s>1} s^{\sigma -\frac{3}{q}+ \frac{1}{p'}-\kappa} e^{ - 2\rho s \left( \beta -\frac{1}{p'} \right)}.
\end{align*}
We recall that $\beta \geq 1/{p'}$ is a necessary condition. Now if $\beta= 1/{p'}$, then the right hand of the inequality above reduced to
\begin{align*}
     C  \sup\limits_{s>1} s^{\sigma -\frac{3}{q}+ \frac{1}{p'}-\kappa },
\end{align*}
which is finite if and only if $$ \sigma + \frac{1}{p'}\leq \kappa + \frac{3}{q}. $$
Therefore, combining all the necessary conditions for $s>1$ case, we derive the following necessary conditions for $q_0\in [1,2)$
\begin{align}\label{nec_sum_[1,2)}
    \sigma<\frac{3}{q}, \qquad \kappa < \frac{n}{p'}, \qquad \text{and} \qquad \delta \geq \frac{1}{p'}-\frac{1}{q_0'}.
\end{align}
 Moreover, for the endpoint case $\delta= \frac{1}{p'}-\frac{1}{q_0'} $, we obtain the following necessary conditions
\begin{align}\label{nec_sum_end[1,2)}
    \sigma<\frac{3}{q}, \qquad \kappa < \frac{n}{p'}, \qquad \text{and} \qquad \sigma + \frac{1}{p'}\leq \kappa+ \frac{3}{q}.
\end{align}
The $q_0=2$ case can be handled similarly. Since for $q_0=2$, $\rho_{q_0}=0$, 
then following a similar calculation as in the previous case and using \eqref{est_v_k,d_2},  we can derive for $s>1$,
\begin{align*}
    \int_0^s  \left(  \frac{v_{\kappa, \delta}(a_t)}{\phi_{{0}}(a_t)} \right)^{-p'} \Delta(t) \,  dt \asymp \int_0^1 t^{-\kappa p'+n-1}+ \int_1^{s} t^{(1-\kappa) p'} e^{2\rho \left( 1- \beta p'\right) t} \, dt,
\end{align*}
where $\beta=\delta+1/2$. By a similar reasoning as before, $\kappa p'<n$ and $\beta \leq 1/{p'}$ are necessary conditions. Assuming the same, it follows 
\begin{align*}
    \left(  \int_0^s  \left(  \frac{v_{\kappa, \delta}(a_t)}{\phi_{0}(a_t)} \right)^{-p'} \Delta(t) \,  dt \right)^{\frac{1}{p'}} \geq C  e^{2\rho \left( \frac{1}{p'}- \beta \right) s} s^{1-\kappa+\frac{1}{p'} },
\end{align*}
which implies that
\begin{equation}
    \begin{aligned}\label{s>1_q0=2}
    &\sup\limits_{s>1}  \left(\int_{0}^{\theta_0/s} u_{\sigma}(\lambda)^q \, d\nu(\lambda) \right)^{\frac{1}{q}} \left(   \int_0^s  \left(  \frac{v_{\kappa, \delta}(a_t)}{\phi_{{0}}(a_t)} \right)^{-p'} \Delta(t) dt 
 \right)^{\frac{1}{p'}} \\
 &\hspace{9cm}   \geq  C   \sup\limits_{s>1} s^{\sigma -\frac{3}{q}+ 1-\kappa+ \frac{1}{p'} } e^{ - 2\rho s \left( \beta -\frac{1}{p'} \right)}.
\end{aligned}
\end{equation}
If $\beta =1/{p'}$,  the inequality above implies that unless 
\begin{align*}
    \sigma +1+ \frac{1}{p'}\leq \kappa + \frac{3}{q},
\end{align*}
the left hand side of \eqref{s>1_q0=2} is infinite. Therefore, we have derived the following necessary condition for the case $q_0=2$ to ensure that the left-hand side of \eqref{s>1_q0=2} is finite
\begin{align}\label{nec_sum_2}
    \sigma<\frac{3}{q}, \qquad \kappa < \frac{n}{p'}, \qquad \text{and} \qquad \delta \geq \frac{1}{p'}-\frac{1}{2}.
\end{align}
 Moreover, for the endpoint case $\delta= \frac{1}{p'}-\frac{1}{2} $, we obtain the following necessary conditions
\begin{align}\label{nec_sum_end2}
    \sigma<\frac{3}{q}, \qquad \kappa < \frac{n}{p'}, \qquad \text{and} \qquad \sigma +1+ \frac{1}{p'}\leq \kappa + \frac{3}{q}.
\end{align}
We conclude the proof our theorem from \eqref{nec_sum_s<1} \eqref{nec_sum_[1,2)}, \eqref{nec_sum_end[1,2)}, \eqref{nec_sum_2}, and \eqref{nec_sum_end2}.
\end{proof}
We would now like to give some remarks that will help us to understand the expectations and scope of such weighted Fourier inequalities.
\begin{remark}
\begin{enumerate}
\item When $n= 2$, a similar result will hold as in the theorem above. However, in these cases, we need to consider the case $\sigma = 2/q$ separately to handle \eqref{int_u_s}. Moreover, when $\sigma \neq 2/q$, one can prove the necessary conditions similarly.

\item In the Euclidean setting $\R^N$, the balance condition $\kappa - \sigma = N ( 1 - \frac{1}{p} - \frac{1}{q} )$ is one of the necessary conditions (see \eqref{k-s=N}). However, in the case of noncompact type symmetric spaces, \eqref{eqn_balance_X} suggests for a given $p,q$, a broader range of $\kappa$ and $\sigma$ can be considered.

\item When $\delta=0$, which means we are considering only polynomial weights, we observe from \eqref{nec_s,k,d} that $p$ has to be less than or equal to $q_0$, if $q_0 \in [1,2]$, and $p\leq q_0'$ if $q_0>2$. Particularly, for $q_0=2$, the condition $p\leq 2$ is necessary. 

\item We would like to mention that the authors in \cite[Corollary 1.8]{KRZ_23} proved Pitt's inequality for polynomial weights in the particular case $q=q_0=2$ and $1\leq p\leq 2$ within the context of noncompact type symmetric spaces of arbitrary rank.  
\end{enumerate}   
\end{remark}
\section{Pitt's  and Paley's inequalities}\label{sec_paley}
In \cite[Theorem 1.3]{DLDS_17}, the authors established Pitt's inequalities using restriction inequalities. Moreover, they provided sufficient conditions of a radial weight {\em without using rearrangement methods}. In this section, our main goal is to prove an analogue of Paley's inequality Theorem \ref{thm_paley} using restriction theorems. This can also be considered an analogue of Pitt's inequality in \cite[Theorem 1.3]{DLDS_17} with $1\leq p=q \leq 2$ in $\X$.

\begin{proof}[\textbf{Proof of Theorem \ref{thm_paley}}]   We consider the measure space $(\X, dx)$ and $(Y, dy):= (\R \times K, \lambda^2 u(\lambda)^{2} (1+|\lambda|)^{n-3}d\lambda dk)$. We define an analytic family of linear operators for   $f \in C_c^{\infty}(\X)$ by
    \begin{align*}
        T_z f(\lambda, k) = \frac{\widetilde{f} (\lambda+z, k) (\lambda+z)}{u(\lambda) (1+|\lambda|)} , \quad \text{where } |\Im z| \leq  \rho.
    \end{align*}
    For  $z =\xi +i\rho$,  
    \begin{align*}
        \|T_z f\|_{L^1(Y)} &= \int_{\R} \int_{K}  \frac {| \widetilde{f} (\lambda+\xi+i\rho, k) ||(\lambda+\xi+i\rho)|} {u(\lambda) (1+|\lambda|)}  \lambda^{2} u(\lambda)^{2} (1+|\lambda|)^{n-3}\,  d\lambda\, dk.
    \end{align*}
     
    By applying Fubini's Theorem and  restriction theorem (Theorem \ref{RS_4.2}), we get
    \begin{align*}
         \|T_z f\|_{L^1(Y)} &= \int_{\R}  \left( \int_{K}   \left| \widetilde{f} (\lambda+\xi+i\rho, k) \right|\, dk \right)  \frac{|(\lambda+\xi+i\rho)|}{   (1+|\lambda|) } \lambda^{2} u(\lambda) (1+|\lambda|)^{n-3} \,d\lambda\\
         & \leq C (1+|\xi|) \|f\|_{L^1(\X)}  \int_{\R}   u(\lambda) \lambda^2 (1+|\lambda|)^{n-3}\, d\lambda.
    \end{align*}
    Then using the estimate \eqref{sharp_c-2} of $|c(\lambda)|^{-2}$, we get 
    \begin{align}\label{Tz_1,1}
         \|T_{\xi+i\rho} f\|_{L^1(Y)} & \leq C (1+|\xi|) \|f\|_{L^1(\X)}  \int_{\R}   u(\lambda) |c(\lambda)|^{-2}\, d\lambda.
    \end{align}
    We note that by our hypothesis, the right-hand side of the inequality above is finite.

      For $z =\xi \in \R$, we get by a change of variable 
    \begin{align*}
        \|T_\xi f\|^2_{L^2(Y)} &= \int_{\R} \int_{K}  \frac{| \widetilde{f} (\lambda+\xi, n) |^2 |\lambda+\xi|^2}{u(\lambda)^2 (1+|\lambda|)^2} |\lambda|^2 u(\lambda)^2 (1+|\lambda|)^{n-3}\,  d\lambda \, dk \\
        &\leq  \int_{\R} \int_{K}   |\widetilde{f} (\lambda+\xi,k) |^2 |\lambda+\xi|^2  (1+|\lambda|)^{n-3}\,  d\lambda\, dk  \\
                &= \int_{\R} \int_{K}  | \widetilde{f} (\lambda, k) |^2 |\lambda|^2   (1+|\lambda-\xi|)^{n-3}\,d\lambda\, dk.
    \end{align*}
    Since we assume $n\geq 3$, we can write 
    \begin{align*}
        (1+|\lambda-\xi|)^{n-3} \leq C  (1+|\xi|)^{n-3} (1+|\lambda|)^{n-3} ,
    \end{align*}
    whence we obtain
    \begin{align*}
         \|T_\xi f\|^2_{L^2(Y)} &\leq C  (1+|\xi|)^{n-3} \int_{\R} \int_{K}  |\widetilde{f} (\lambda, k) |^2 |\lambda|^2  (1+|\lambda|)^{n-3}\,  d\lambda \,dk.
    \end{align*}
Using \eqref{sharp_c-2} in the inequality above, and followed by the Plancherel theorem, we derive
\begin{align}\label{Tz_2,2}
    \|T_\xi f\|^2_{L^2(Y)}\leq C (1+|\xi|)^{n-3} \int_{\R} \int_{K}  |\widetilde{f} (\lambda, k) |^2 |c(\lambda)|^{-2}\,  d\lambda\,  dk =  C (1+|\xi|)^{n-3}  \|f\|_{L^2(\X)}^2.
\end{align}
By applying Stein analytic interpolation theorem, it follows from \eqref{Tz_1,1} and \eqref{Tz_2,2} that for $1\leq p\leq 2$
\begin{align}\label{Lp_Tf}
       \left(  \int_{\R} \int_{K} \frac{ | \widetilde{f}(\lambda+i\rho_p, k)|^p |\lambda+i\rho_p|^p}{(1+|\lambda|)^p}  
  u(\lambda)^{2-p} \lambda^2 (1+|\lambda|)^{n-3} \,  d\lambda \, dk   \right)^{\frac{1}{p}} \leq  C_p \|u\|_{L^1\left( |c(\lambda)|^{-2}\right)}^{\frac{2}{p}-1} \|f\|_{L^p(\X)},
    \end{align}
 Now using \eqref{sharp_c-2} the estimate of $|c(\lambda)|^{-2}$ and the following  
    \[   \frac{|\lambda+i\rho_p|}{(1+|\lambda|)}\asymp  1,\]
     for $1\leq p<2$, we can write from \eqref{Lp_Tf},
    \begin{align}
         \left(  \int_{\R} \int_{K}  \left| \widetilde{f}(\lambda+i\rho_p, k)\right|^p u(\lambda)^{2-p} |c(\lambda)|^{-2}\, d\lambda\,  dk \right)^{\frac{1}{p}} \leq  C_p \|u\|_{L^1( |c(\lambda)|^{-2})}^{\frac{2}{p}-1} \|f\|_{L^p(\X)},
    \end{align}
    holds for all $p\in [1,2]$.
\end{proof}
\begin{remark}
\begin{enumerate}
\item We observe that if $u$ is a non-negative function such that  
        \begin{align*}
         \|u\|_{L^1(|c(\lambda)|^{-2})} =  \int_{\R}u(\lambda) |{c(\lambda)}|^{-2}\, d\lambda <\infty,
        \end{align*}
        then using \cite[Theorem 1.1]{KRS_10}, we have  for $q_0 \in [p,p']$
\begin{align*}
 \int_{\R} \left(\int_{K}\left|\widetilde{f}(\lambda+i\rho_{q_0},k)\right|^{q_0} dk \right)^{\frac{1}{q_0}}  u(\lambda) |{c(\lambda)}|^{-2}\,d\lambda  &\leq C_{p,q_0} |\|u\|_{L^1(|c(\lambda)|^{-2})}\|f\|_{L^{p,1}(\X)}.
\end{align*}

\item  For $1\leq p\leq 2$ and $q_0\in [p,p']$, if $u \in L^p(\R, |c(\lambda)|^{-2})$, then \eqref{pitt_nrad_int} is true for $q=1$ and $v\equiv 1$. Indeed, by the Hausdorff-Young inequality \eqref{thm_RS_HYinq}, 
\begin{align*}
   &  \int_{\R}\left( \int_K\left|  \widetilde{f}(\lambda +i\rho_{q_0},k) \right|^{q_0} dk\right)^{\frac{1}{q_0}} u(\lambda) |{c(\lambda)}|^{-2}\, d\lambda \\
   & \leq \|u\|_{L^p(|c(\lambda)|^{-2})} \left(\int_{\R}\left( \int_K\left|  \widetilde{f}(\lambda +i\rho_{q_0},k) \right|^{q_0} dk\right)^{\frac{p'}{q_0}}  |{c(\lambda)}|^{-2}\, d\lambda\right)^{\frac{1}{p'}} \\
   & \leq C_{p,q_0}\|u\|_{L^p(|c(\lambda)|^{-2})} \|f\|_{L^{p}(\X)}.
\end{align*}
    \end{enumerate}
\end{remark}

\section*{Acknowledgments}
PK is supported partially by SERB MATRICS, grant number MTR/2021/000116. TR is supported by the FWO Odysseus 1 grant G.0H94.18N: Analysis and Partial Differential Equations, the Methusalem program of the Ghent University Special Research Fund (BOF) (Grant number 01M01021). MS is supported by an institute Ph.D. fellowship at IIT Guwahati.
\bibliographystyle{alphaurl}
\bibliography{Ref_Pitts}

\end{document}